\documentclass[10pt,a4paper]{amsart}
\usepackage{amssymb,amsmath,amsfonts,amscd}
\usepackage{lmodern}
\usepackage[utf8]{inputenc} 
\usepackage[T1]{fontenc}    
\usepackage{url}            
\usepackage{booktabs}       
\usepackage{fancyhdr}
\usepackage{indentfirst}
\usepackage{graphicx}
\usepackage{subfigure}
\usepackage{wrapfig}
\usepackage{newlfont}
\usepackage{color}
\usepackage{latexsym}
\usepackage{csquotes}
\usepackage{amsthm}
\usepackage{mathtools}
\usepackage{nicefrac}
\usepackage{epstopdf}
\usepackage{caption}
\usepackage{bbm}
\usepackage[shortlabels]{enumitem}
\usepackage{tikz-cd}
\usepackage{microtype}      
\usepackage{lipsum}
\usepackage{comment}
\graphicspath{ {./images/} }

\newcommand{\Lsu}[2]{\bigwedge\nolimits^{#1}{#2}}
\newcommand{\Lgiu}[2]{\bigwedge\nolimits_{#1}{#2}}

\def\g{\mathfrak{g}}

\usepackage{hyperref}

\newcommand{\R}{\mathbb R}

\newcommand{\G}{\mathbb G}

\newcommand{\BL}[2]{BL^{{#1},{#2}}
(\G)}

\newcommand{\BLh}[3]{BL^{{#1},{#2}}
(\G, E_0^{#3})}

\newcommand{\scal}[2]{\langle {#1} , {#2}\rangle}

\newcommand{\Scal}[2]{\langle {#1} \vert {#2}\rangle}

\newcommand{\ccheck}{{\vphantom i}^{\mathrm v}\!\,}

\newcommand{\mc}{\mathcal }

\newcommand{\eps}{\varepsilon}

\newcommand{\N}{\mathbb N}

\newcommand{\e}{\mathrm{Euc}}

\def \N {\mathbb{N}}
\def \R {\mathbb{R}}

\def \H {\mathbb{H}}

\def \d {\mathrm{d}}

\def \de {\partial}

\def \g {\mathfrak{g}}

\def \e {\varepsilon }

\def \ker {\mathrm{Ker}}
\def \Im {\mathrm{Im}}
\def \dsy {\displaystyle}
\def \Span {\mathrm{span}}

\def \G {\mathbb{G}}

\newcommand*{\checkV}[2][-3mu]{\ensuremath{\mskip1mu\prescript{\smash{\mathrm v\mkern#1}}{}{\mathstrut#2}}}

\newtheorem{teo}{Theorem}[section]
\newtheorem{prop}[teo]{Proposition}
\newtheorem{coro}[teo]{Corollary}
\newtheorem{lemma}[teo]{Lemma}
\newtheorem{defi}[teo]{Definition}

\newtheorem{proposition}[teo]{Proposition}
\newtheorem{remark}[teo]{Remark}

\newtheorem{oss}[teo]{Remark}

\title{$L^p$-Hodge decomposition 
and global integral estimates on the Cartan group}

\author{
 Annalisa Baldi, 
 Alessandro Rosa 
}

\begin{document}

\begin{abstract}
The study of Sobolev and Poincaré inequalities for differential forms in Carnot groups and in the more general sub-Riemannian setting is still an open problem in its full generality. One may conjecture that, for general Carnot groups, these inequalities are expressed in terms of suitable graded Lebesgue norms. In recent years, many results have been obtained, both in the Euclidean setting and in the Heisenberg groups, as well as for contact manifolds with bounded geometry. There are also some results for general Carnot groups; however, these do not cover the problem in its full generality.

In this paper, we consider a particular Carnot group, the so-called Cartan group (a free Carnot group, of step $3$ with $2$ generators), which provides a natural testing ground for these questions, since its step-three structure already exhibits several phenomena that do not occur in the Heisenberg groups.
In this setting, we are able to prove global Poincaré and Sobolev-Gaffney  inequalities for differential forms.  With the aim of obtaining sharp estimates,  we replace the de Rham complex of differential forms with the Rumin complex. The case $p>1$ is carried out after establishing an $L^p$-Hodge decomposition with homogeneous Sobolev classes. We are able to consider also the endpoint case $p=1$; however, as in Euclidean setting, when $p=1$, the operator we deal with provides only weak-type estimates which do not yield a Hodge decomposition analogous to the case $p>1$. Therefore, in this situation the proof follows a different approach, relying on a recent result proved in \cite{BT}, 
\end{abstract}




\maketitle
\tableofcontents

\section{Introduction}

In $\R^n$, if
\begin{eqnarray*}
1\leq p,q< \infty,\quad \frac{1}{p}-\frac{1}{q}=\frac{1}{n},
\end{eqnarray*}
the Sobolev inequality  states that there exists a constant $C=C(p,q,n)>0$ such that if  $u$ is  a compactly supported function
(i.e.\,a compactly supported 0-form),
then 
\begin{eqnarray*}
\|u\|_{L^q} \leq C\|du\|_{L^p},
\end{eqnarray*} 
where $d$ denotes the exterior differential (hence $du$ is a $1$-form).
The Poincar\'e inequality is a variant formulated for functions $u$  that are not necessarily compactly supported. Poincar\'e and Sobolev inequalities can also be formulated   for differential forms, including  in the  general setting of  Riemannian manifolds. If  $M$ is a Riemannian manifold, we   say that a  $(p,q)$-Poincar\'e$(h)$ inequality holds on $M$, if 
there exists a positive constant $C=C(M,p,q)$ such that for every exact $h$-form $\alpha\in L^p$, there exists a $(h-1)$-form $\phi\in L^q$ such that $d\phi=\alpha$ and
\begin{eqnarray*}
\|\phi\|_{L^q} \leq C\,\|\alpha\|_{L^p}.
\end{eqnarray*} 

The equation $d\phi=\alpha$ is understood in a distributional sense. Analogous definitions have been studied in several sub-Riemannian settings where Rumin's complex replaces the de Rham complex, notably on Heisenberg groups and contact manifolds
 (see e.g. \cite{BFP3,BFP2,BFP5,BFP6,BTT}). In this paper, we focus on another specific sub-Riemannian setting, namely the Cartan group.

\medskip

Roughly speaking, a sub-Riemannian structure on a manifold $M$ is   determined by a subbundle $H \subset TM$ of the tangent bundle which defines the admissible directions at any points of $M$.
In recent years, sub-Riemannian geometries have attracted considerable attention from several perspectives, including differential geometry, geometric measure theory, subelliptic partial differential equations, complex analysis, optimal control, mathematical models in neuroscience and robotics.
Endowing each fiber $H_x$ with an inner product naturally induces a  distance on $M$, defined as the infimum of the Riemannian lengths of curves $\gamma$ such that $\dot{\gamma}(t) \in H_{\gamma(t)}$. Such distance is usually called Carnot--Carath\'eodory distance. 
Among sub-Riemannian spaces, Carnot groups  play a role analogous to that of Euclidean spaces $\R^N$ in Riemannian geometry. They can be seen as rigid ``tangent'' models for general sub-Riemannian manifolds.  We can identify a Carnot group $\mathbb{G}$ with a Lie group $(\mathbb{R}^N,\cdot)$ whose Lie algebra $\mathfrak{g}$ is nilpotent and admits a stratification of step $\kappa$. More precisely, there exist linear subspaces $V_1,\dots,V_\kappa$, called the layers of the stratification, such that
\[
\mathfrak{g} = V_1 \oplus \cdots \oplus V_\kappa, \qquad
[V_1,V_i] = V_{i+1}, \qquad
V_\kappa \neq \{0\}, \qquad
V_i = \{0\} \ \text{for } i > \kappa,
\]
where $[V_1,V_i]$ denotes the subspace of $\mathfrak{g}$ generated by commutators $[X,Y]$ with $X \in V_1$ and $Y \in V_i$. The first layer $V_1$, called the horizontal layer,  plays a fundamental role, since it generates the whole Lie algebra $\mathfrak{g}$ through iterated commutations. In any non-commutative Carnot group  the so-called homogeneous dimension $Q:=\sum_{i=1}^\kappa i\,\text{dim}\,V_i$ -- which coincides with the Hausdorff dimension of the group viewed as a manifold -- is strictly greater than the topological dimension $N$.

As already quoted above, Sobolev and Poincaré inequalities for differential forms on Heisenberg groups and contact manifolds have been widely studied in the last decade. Unfortunately, the study of these inequalities on general Carnot groups remains largely open. In this paper, we study another important example of a Carnot group - the Cartan group - which provides a natural testing ground for these questions, since its step-three structure already exhibits several phenomena that do not occur in step-two groups such as the Heisenberg group. The Cartan group is
  the $5$-dimensional nilpotent Carnot group of step $\kappa=3$ with $2$ generators.
 More precisely, let us consider the vector fields
\[
(X_1,\,\,X_2,\,\,X_3 = [X_1, X_2],\,\,X_4 = [X_1, X_3],\,\,X_5 = [X_2, X_3]).
\]
The group can be identified  with $\R^5$ through exponential coordinates.
The stratification of the Lie algebra induces a family of non-isotropic dilations $\{\delta_\lambda\}_{\lambda>0}$ on the group such that, for any $x=(x_1,\ldots,x_5)\in\R^5$,
$$\delta_\lambda(x)=(\lambda x_1,\lambda x_2,\lambda^2 x_3, \lambda^3 x_4, \lambda^3 x_5).$$
In this case, the homogeneous dimension is $Q=10$, whereas the topological dimension is $N=5$.
This group provides a particular example of a smooth $5$-manifold $M$, whose smooth subbundle $H \subset TM$  has rank two and it is a bracket generating distribution with growth vector $(2,3,5)$. Since Cartan's paper \cite{cartan1910}, these kinds of manifolds have attracted increasing interest, as they arise in several contexts. For instance a $(2,3,5)$ manifold of this type provides an example of non-holonomic mechanical system which describes the rolling of one surface over another
 in $\R^3$, without slipping and twisting (see \cite{BryantHsu}, Section 4.4). The Cartan group also provides an example of conformal geometry naturally associated with ordinary differential equations \cite{CapNeusser,Nurowski}. In recent years, hypoelliptic Hodge-Laplace operators on the Cartan group  \cite{Haller2022,haller2023regularized,haller2023analytic} have been studied with applications to heat kernel expansions and analytic torsion. Recently, in \cite{BT} hypoelliptic Hodge-Laplace operators on the Cartan group have been considered and applied to
the study of div-curl systems in the group. These operators are associated to the so-called Rumin complex $$(E_0^\bullet,d_c),$$
a subcomplex of the de Rham complex introduced by Rumin first for contact manifolds and later for the more general sub-Riemannian setting (see e.g.\,\cite{rumin_jdg}, \cite{rumin_cras}, \cite{rumin_grenoble}, \cite{rumin_palermo} -- in Section 2.5 we provide a short introduction to the Rumin  complex).

Indeed, when dealing with a Carnot group, the Rumin complex better reflects its geometry, since the stratification and the dilations of the group play a prominent role, for instance in obtaining sharp integral estimates. Nevertheless, as stressed above, the treatment of these integral inequalities in general Carnot groups is still far from being fully understood. In \cite{pansu_rumin18}, the authors study $L^p$-Poincaré inequalities for Rumin differential forms on general Carnot groups, using suitable zero-order pseudodifferential Laplacians. However, they obtain sharp estimates only for a few specific Carnot groups, while in the remaining cases the estimates are not sharp and need stronger assumptions. Moreover, their arguments do not cover the case $p=1$.
One of the aims of this paper is to obtain sharp $L^p$-estimates on the Cartan group. In particular, the inequalities for 
 $p=1$ can be established. We also prove Sobolev-Gaffney inequalities with sharp estimates.

The study of the connection between cohomology and Rumin harmonic forms naturally involves the Hodge decomposition. Starting from the works of Helmholtz \cite{helmholtz}, Hodge \cite{hodge1,hodge2}, and de Rham \cite{deRham1,deRham2}, the Hodge decomposition has been widely used in the study of problems arising, for instance, from the Navier–Stokes equations \cite{leray} and in connection with PDEs, quasiconformal maps, and potential theory on Riemannian manifolds (see \cite{schwarz,duff_spencer,ISS,XiDong2}).

  A first example of Hodge decomposition for $L^2$-forms on complete Riemannian manifolds was proved by Kodaira \cite{kodaira} in 1949, while more recently the $L^p$-Hodge decomposition, with $1<p<\infty$, established  in \cite{XiDong}.

In the sub-Riemannian setting, and in particular on Carnot groups, the study of Hodge decompositions with Sobolev spaces presents additional difficulties due to the lack of ellipticity of the relevant differential operators. 
A first example of an $L^p$-Hodge decomposition on sub-Riemannian manifolds was obtained in \cite{BR}, in the setting of  contact sub-Riemannian  manifolds. In this context, the Rumin complex provides a natural replacement for the de Rham complex and allows one to recover a suitable hypoelliptic Hodge theory. 

In this paper as well, we  establish an 
$L^p$-Hodge decomposition for Rumin-differential forms, with homogeneous Sobolev spaces, when $1<p<\infty$. 
  Consider, for instance, a Rumin $1$-form. If we denote by $\G$ the Cartan group and by $\delta_c$ the $L^2$-formal adjoint of $d_c$, we have the following decomposition
\begin{equation}\label{eq: 2 gennaio prima}
L^p(\G,E_0^1)=d_c BL^{1,p}(\G)\oplus\delta_c BL^{3,p}(\G,E_0^{2})\,,
\end{equation}
where $BL^{s,p}(\G)$ denote the homogeneous Sobolev spaces (see Definition \ref{defi.BLG}). Our main results on the Hodge decomposition are summarized in Theorem \ref{Lphodge2}, that
 can be used to obtain $L^p$-Poincaré inequalities for $1<p<\infty$, as proved in Section \ref{Lpsection}. Below we state the $L^p$-Poincaré inequality   for the case of $1$-forms, in order to illustrate the general result stated in Theorem  \ref{teo.globalLpPoincineq}.

\noindent\textit{
Suppose $1<p<Q$ and let $q\geq p$ satisfy $\dsy\frac{1}{q}=\frac{1}{p}-\frac{1}{Q}$. 
Then there exists a constant $C=C(Q,p)>0$ such that for every $d_c$-closed $1$-form $\alpha\in L^p(\G,E_0^1)$, there exists $\phi\in L^q(\G)$ such that $$d_c\phi=\alpha\quad\quad\text{and}\quad\quad\|\phi\|_{L^{q}(\G)}\leq C\|\alpha\|_{L^p(\G,E_0^1)}.$$
}

\medskip

When dealing with the case $p=1$, we cannot expect to have an $L^1$-result analogous to \eqref{eq: 2 gennaio prima}. Indeed,  even in the Euclidean setting $\R^n$ the $L^1$-decomposition with Sobolev classes is false (see  \cite{BaldoOrlandi98}).
One can only obtain a decomposition in the sense of distributions with estimates in suitable weak Lebesgue spaces. Therefore, proving a sharp Poincaré inequality in the endpoint case $p=1$ requires different ideas, closely related to the div-curl estimates proved in \cite{BT}. One of our main results is the  $L^1$-Poincaré inequality proved in Section \ref{Lpsection}. We state it here only for $1$-forms
 and refer to  Theorem \ref{P1} for the general case of forms of arbitrary degree.

\noindent\textit{
There exists a constant $C>0$ such that, for every $d_c$-closed $1$-form $\alpha\in L^1(\G,E_0^1)$, there exists $\phi\in L^{Q/(Q-1)}(\G)$ such that
\begin{equation}\label{p=1 intro}
    d_c\phi=\alpha\quad\quad\text{and}\quad\quad\|\phi\|_{L^{Q/(Q-1)}(\G)}\leq C\|\alpha\|_{L^1(\G,E_0^1)}.
\end{equation}}

It is interesting to observe that the previous Poincaré inequalities easily implies the classical Poincaré inequalities for functions, already well known in any Carnot group (see e.g. \cite{FLW_grenoble,FGaW,CDG,gromov,Lu94}).
Indeed, the classical Poincar\'e inequality 
for functions says that if $1\le p\le Q$, for any (say) Lipschitz continuous function $u$ such that $\nabla_\G u\in L^p(\G)$ there exists a constant $c_u$ such that 
$$
\|u-c_u\|_{L^{pQ/(Q-p)}}\leq C(Q,p)\,\|\nabla_\G u\|_{L^p}.
$$
 where $\nabla_\G$ is the intrinsic gradient (see Section \ref{sec:BL2} below). For a 0-forms $u$, $d_cu$ can be identified with $\nabla_\G u$.  Observe that $\alpha:=d_c u$ is a closed form,
so that there exists $\phi$ in $L^{Q/(Q-1)}(\G)$  such that \eqref{p=1 intro} holds. Then $ u - \phi  = c_u$ (since $u-\phi $ is closed). Therefore
$$
\|u-c_u\|_{L^{pQ/(Q-p)}} = \|\phi\|_{L^{pQ/Q-p}}\leq C\,\|d_cu\|_{L^{p}} \le\, C \|\nabla_\G u\|_{L^{p}}.
$$

\bigskip 

The paper is organized as follows.
Section \ref{sec:prelres} reviews basic definitions and properties of Carnot groups and briefly introduces the Rumin complex. In particular, Section \ref{sec:kernels} recalls the notions of kernels and convolutions, as well as Folland-Stein inhomogeneous Sobolev spaces, together with their main properties.

Section \ref{sec:BL2} introduces what we call Beppo Levi homogeneous Sobolev spaces and  proves a characterization of the corresponding norms in terms of horizontal derivatives.

Section \ref{ex.Cartanesempio} is devoted to a brief introduction to the Rumin complex on the Cartan group. In particular, we analyze in detail the definition and properties of the homogeneous left-invariant Laplacians $\Delta_{R,h}$ introduced in \cite{BT}.

In Section \ref{sec:LpHodge}, we prove an $L^p$-Hodge decomposition for $1<p<\infty$. The proof relies on the Hodge decomposition for smooth forms, which is also established in this section, together with $L^p$-continuity estimates for kernels. Moreover, we obtain a weak $L^1$-Hodge decomposition based on estimates in weak-Lebesgue spaces.

Section \ref{sec.PoincareP} is devoted to Poincar\'e inequalities. We first prove $L^p$-Poincar\'e inequalities for $1<p<\infty$, as a consequence of the Hodge decomposition. In contrast, we establish $L^1$-Poincar\'e inequalities based on Gagliardo-Nirenberg estimates obtained in \cite{BT} and a suitable result for homogeneous kernels on shells proved in \cite{BFP3}.

Finally, we also establish Sobolev–Gaffney inequalities, providing intrinsic characterizations of Sobolev spaces of Rumin forms in terms of the Rumin differential $d_c$ and codifferential $\delta_c$,
 which are naturally adapted to the stratification of the group. In the setting of Heisenberg groups, in bounded regions, Gaffney inequalities have been considered in \cite{franchi_montefalcone_serra}, and for general contact manifolds with bounded geometry in \cite{BTT}.

\section{Preliminary results}\label{sec:prelres}

\subsection{Carnot groups}
We recall that a {Lie group} is a smooth manifold $\G$ endowed with a binary operation $\cdot$ such that $(\G,\cdot)$ is a group and the maps $$\G\times\G\ni (x,y)\mapsto x\cdot y\in\G\quad\text{and}\quad\G\ni x\mapsto x^{-1}\in\G$$
are both smooth. We set $N$ as the topological dimension of $\G$.

For $x\in\G$ fixed, we denote by 
\begin{equation}\label{eq.trasl}
\tau_x:\G\to\G,\quad\tau_x(g):=x\cdot g
\end{equation}
the {left-translation by $x$}. A smooth vector field $X$ on $\G$ is called \textit{left-invariant} if $$X(f\circ\tau_x)=(Xf)\circ\tau_x,\quad\text{for every $x\in\G$, $f\in C^{\infty}(\G)$,}$$
or, equivalently, if $$\d_y\tau_x(X_y)=X_{x\cdot y},\quad\text{for every $x,y\in\G$.}$$
We denote by $\g$ the \textit{Lie algebra of $\G$} that is the vector space of all the left-invariant vector fields on $\G$.

A \textit{Carnot group $\G$ of step $\kappa$} is a connected, simply connected, nilpotent Lie group whose Lie algebra is stratified, i.e.\,there exist nontrivial linear subspaces $V_1,\ldots V_\kappa$ such that
$$\g=V_1\oplus\cdots\oplus V_\kappa$$
and
$$V_{i+1}=\left[V_1,V_i\right],\text{\quad for every $i=1,\ldots,\kappa-1,$\quad\quad $V_\kappa\neq\{0\}$\quad and \quad $\left[V_1,V_\kappa\right]=0,$}$$
where $[V_1,V_i]$ is the subspace of $\g$ generated by the commutators $[X,Y]$ with $X\in V_1$ and $Y\in V_i$.
We set $m_j:=\dim(V_j)$ and $h_j:=\sum_{i=0}^j m_i$, for every $j=0,\ldots,\kappa$.

Through the exponential map, we can identify $p$ with the $N$-uple $(p_1,\ldots,p_N)\in\R^N$ and we identify $\G$ with $(\R^N,\cdot)$ where the explicit expression of the group operation $\cdot$ is determined by the Campbell-Baker-Hausdorff-Dynkin formula (see \cite{BLU}).
In exponential coordinates,
the unit element $e$ of $\G$ is $e=(0,\dots,0)$.

The first layer $V_1$ will be called \textit{horizontal layer}; a left-invariant
vector field in $V_1$, identified with a differential operator, will be
called an \textit{horizontal derivative}.
We denote by $\{X_1,\dots,X_{m_1}\}$ a basis of $V_1$.
Since $\g$ is a finite dimensional space, we endow $\g$ with an inner product $\langle\cdot,\cdot\rangle$, so that the adapted basis $\{X_1,\ldots,X_N\}$ is an orthonormal basis of $\g$.

For any $\lambda >0$, the
{\it dilation} $\delta_\lambda:\G\to\G$, is defined as
\begin{equation}\label{dilatazioni}
\delta_\lambda(x_1,...,x_N)=
(\lambda^{d_1}x_1,...,\lambda^{d_N}x_N),
\end{equation} where $d_i\in\N$ is called the \textit{homogeneity of
the variable} $x_i$ in
$\G$ (see \cite{folland_stein} Chapter 1).
We denote by $Q$ the \textit{homogeneous dimension
of $\G$} defined by
\begin{equation}\label{dimensione}
Q:=\sum_{i=1}^\kappa i\,\text{dim}\,V_i.
\end{equation}

Connected, simply connected, nilpotent Lie groups are locally compact groups and the  Haar measure coincides with the {Lebesgue measure} $\mathcal{L}^N$. Thus, we have $\mathcal{L}^N(g\cdot E)=\mathcal{L}^N(E)$, for every $E\subseteq\G\cong\R^N$ measurable set, for every $g\in\G$.

In this paper, we denote by $ \rho(\cdot)$ a homogeneous norm, smooth outside the origin, that
induces a genuine distance on $\G$ as in \cite{BT}. We associate the distance
$$d(x,y):=\rho(y^{-1}\cdot x).$$
We have that $d$ is homogeneous of degree $1$ with respect to group dilations and it is left-invariant, i.e.\,
$$d(z\cdot x, z\cdot y)=d(x,y)\quad\text{and}\quad d(\delta_\lambda(x),\delta_\lambda(y))=\lambda d(x,y),\quad\forall\,x,y,z\in\G,\lambda>0.$$

\bigskip

\subsection{Convolutions, kernels and Sobolev spaces in Carnot groups}\label{sec:kernels}
 Let $f\in \mathcal{D}(\G)$ and $g\in L^1_{\mathrm {loc}}(\G)$, then we define the convolution on $\G$ between $f$ and $g$ as
$$(f\ast g)(x):=\int_{\G}\,f(y)g(y^{-1}\cdot x)\d y,\quad\quad\forall\,x\in\G.$$
If $g$ is a smooth function and $L$ is a left-invariant differential operator, then
$$L(f\ast g)=f\ast Lg.$$
We denote
    by $\ccheck f$ the function defined by $\ccheck f(p):=
    f(p^{-1})$, and, if $u\in\mc D'(\G)$, then $\ccheck u$
    is the distribution defined by $\Scal{\ccheck u}{\phi}
    :=\Scal{u}{\ccheck\phi}$ for any test function $\phi$.
Also, if $u,v\in\mathcal{D}'(\G)$ and at least one of them has compact support, it is well posed the convolution between $u$ and $v$, and
$$\langle u\ast v|\varphi\rangle=\langle v|\checkV{u}\ast\varphi\rangle=\langle u|\varphi\ast\checkV{v}\rangle,\quad\quad\forall\,\varphi\in \mathcal{D}(\G).$$
We will use the following notations for higher-order derivatives: if $J=(j_1,\ldots,j_{N})$ is a multi-index, we set
\begin{equation}\label{derivate}
X^{J}=X_1^{j_1}\cdots X_N^{j_N},
\end{equation}
and
$$|J|:=j_1+\cdots+j_{N},\quad\quad d(J):=d_1 j_1+\cdots+d_N j_{N},$$
that are, respectively, the order of the differential operator $X^{J}$ and its degree of homogeneity with respect to group dilations. By the \textit{Poincaré-Birkhoff-Witt theorem} (see \cite{ricci}), we know that the differential operators $(X^{J})_J$ form a basis of the algebra of left-invariant differential operators in $\G$.
Moreover, if $v\in \mathcal{D}'(\G)$ and $u\in\mathcal{E}'(\G)$, we have that, for every $\varphi \in \mathcal{D}(\G)$,
\begin{align}\label{eq.convspostcose}
\langle (X^I u)\ast v|\varphi\rangle&=\langle X^I u| \varphi\ast \checkV{v}\rangle=(-1)^{|I|}\langle u| \varphi\ast (X^I\,\checkV{v})\rangle=(-1)^{|I|}\langle u\ast\checkV{X^I}\,\checkV{v}|\varphi\rangle.
\end{align}

We recall that, if $u\in\mathcal{D}'(\G)$, then $u$ is called \textit{homogeneous distribution of degree $\mu\in\R$} if
$$\langle u\circ \delta_\lambda|\varphi\rangle:=\langle u|\lambda^{-Q}\varphi\circ \delta_{\lambda^{-1}}\rangle=\lambda^{\mu}\langle u|\varphi\rangle,\quad\forall\,\varphi\in \mathcal{D}(\G), \lambda>0.$$
If $K\in\mathcal{D}'(\G)$,  the convolution operator $u\mapsto u\ast K$ from $\mathcal{D}(\G)$ to $C^\infty(\G)$
is a left-invariant operator. We refer to $K$ as the \textit{kernel} of the operator.
We say that $K$ is a \textit{kernel of type $\mu$} if $K$ is a homogeneous distribution of degree $\mu-Q$, that is smooth outside the origin, and the convolution operator with the kernel $K$ is called \textit{operator of type $\mu$}.\bigskip

\begin{lemma}[see Lemma $3.11$ in \cite{BFP3}]\label{lemma.checkconvkernel}
If $K$ is a kernel of type $\mu\in(0,Q)$, $f\in L^1(\G)$ and $\varphi\in \mathcal{D}(\G)$, then
$$\langle f\ast K|\varphi\rangle=\langle f|\varphi\ast \checkV{K}\rangle.$$
The result is still true if we assume that $K\in L^1_{loc}(\G)$ and $f$ is compactly supported.
\end{lemma}
We recal now a key result, proved in \cite{BFP3}. If $P$ is the operator of convolution with a kernel of type $\mu>0$, and $f\in L^1$, then the $L^1$-norm of $Pf$ on shells $B(e,2R)\setminus B(e,R)$ is $O(R^\mu)$. If furthermore $f$ has vanishing average, this can be improved to $o(R^\mu)$, as stated in the following theorem.

\begin{teo}[see Theorem $3.13$ in \cite{BFP3}]\label{teo.avesterrorG}
If $K$ is a kernel of type $\mu\in (0,Q)$, then for every $f\in L^1(\G)$ such that $\dsy\int_\G f=0$, we have
$$R^{-\mu}\int_{B(e,2R)\setminus B(e,R)}|f\ast K| \xrightarrow{R\to+\infty} 0.$$
\end{teo}

\begin{prop}[Proposition $1.11$ in \cite{folland}]\label{teo.contofconvtypeop}
Suppose $0<\mu<Q$, $1<p<Q/\mu$ and $\dsy\frac{1}{q}=\frac{1}{p}-\frac{\mu}{Q}$. Let $K$ be a kernel of type $\mu$. If $u\in L^p(\G)$, then the convolutions $u\ast K$ and $K\ast u$ exist a.e.\,and are in $L^q(\G)$. Moreover, there is a constant $C=C(p,\mu)>0$ such that
$$\|u\ast K\|_{L^q(\G)}\leq C\|u\|_{L^p(\G)}\quad\text{and}\quad\|K\ast u\|_{L^q(\G)}\leq C\|u\|_{L^p(\G)}.$$
\end{prop}
\medskip
We recall the notion of \textit{inhomogeneous Folland-Stein Sobolev space} (see \cite{folland_stein}).

\begin{defi}\label{defi.SobolevHn}
Let $1\leq p\leq \infty$ and $s\in\N$. 
Then, we define $W^{s,p}(\G)$ as the space of all $f\in L^p(\G)$ such that
$$X^Jf\in L^p(\G),\quad\text{for all multi-indices $J$ with $d(J)\leq s$},$$
endowed with the norm
$$\|f\|_{W^{s,p}(\G)}:=\sum_{d(J)\leq s}\|X^J f\|_{L^p(\G)}.$$
\end{defi}

Folland-Stein Sobolev spaces have the following properties (see  \cite{folland}, \cite{FSSC_houston}, \cite{BLU}).
\begin{teo}\label{teo.denseSobspHn}
Let $1\leq p\leq \infty$ and $s\in\N$. Then,
\begin{enumerate}
\item $W^{s,p}(\G)$ is a Banach space;
\item if $p<\infty$, $W^{s,p}(\G)\cap C^\infty(\G)$ and $\mathcal{D}(\G)$ are dense in $W^{s,p}(\G)$;
\item if $1<p<\infty$, then $W^{s,p}(\G)$ is reflexive;
\item the Definition \ref{defi.SobolevHn} is independent from the choice of the basis $\{X_1,\ldots, X_N\}$ (see Corollary 4.14 in \cite{folland}).
\end{enumerate}
\end{teo}

\begin{prop}[Proposition 1.9 and Theorem 4.9 in \cite{folland}]\label{folland zero}If $K$ is a kernel of type $0$, $1<p<\infty$, $s\in\N$, then the mapping $u\mapsto u\ast K$ defined for $u\in \mathcal{D}(\G)$ extends to a bounded operator on $W^{s,p}(\G)$
\end{prop}

\begin{prop}\label{prop.kernelopcheck}
Let $K\in\mathcal{D}'(\G)$ be a kernel of type $\mu$. Then,
\begin{enumerate}
\item $\checkV{K}$ is again a kernel of type $\mu$;
\item $X_\ell K, K X_\ell$ are kernels of type $\mu-1$ for any horizontal derivative $X_\ell$, with $\ell=1,\ldots,m_1$;
\item if $\mu>0$, then $K\in L^1_{loc}(\G)$.
\end{enumerate}
\end{prop}

\begin{lemma}[see Lemma $3.3$ in \cite{BFP3}]\label{lemma.estimatekernelpol}
Let $K$ be a kernel of type $\mu>0$ and let $\varphi\in \mathcal{D}(\G)$ be a test function. Then, $\varphi\ast K$ is smooth on $\G$.

If $R$ is an homogeneous polynomial of degree $\ell\geq 0$, in the horizontal derivatives, we have 
$$R(\varphi\ast K)(x)=O(|x|^{\mu-Q-\ell})\quad\quad\text{as $x\to\infty$}.$$
In particular, if $K$ is a kernel of type $\mu<Q$, then both $\varphi\ast K$ and all its derivatives belong to $L^\infty(\G)$.
\end{lemma}

Let $1\leq p\leq\infty$. We recall that the \textit{weak-$L^p$ space} $L^{p,\infty}(\G)$ is defined as the space of all measurable functions $f$ on $\G$ such that
$$\sup_{t>0}\,\lambda_f^p(t)<\infty,$$
where $\lambda_f(t):=\mathcal{L}^N(\{|f|>t\})$ is the distribution function of $f$.

\begin{lemma}[see Lemma $3.9$ in \cite{BFP3}]\label{lemma.stimaconvL1supp}
Let $K$ be a kernel of type $\mu\in (0,Q)$. Then, for every $f\in L^1(\G)$, we have $f\ast K\in L^{Q/(Q-\mu),\infty}(\G)$ and there exists $C>0$ such that
$$\|f\ast K\|_{L^{Q/(Q-\mu),\infty}(\G)}\leq C\|f\|_{L^1(\G)}.$$
In particular, if $1\leq p<Q/(Q-\mu)$, then $f\ast K\in L^p_{loc}(\G)\subseteq L^1_{loc}(\G)$.

Moreover, if $\psi\in \mathcal{D}(\G)$ such that $\psi\equiv 1$ in a neighborhood of the origin, we have the same results if we replace $K$ with $\psi K$ or $(1-\psi)K$.
\end{lemma}

\medskip

\subsection{Beppo Levi-Sobolev spaces $BL^{s,p}(\mathbb G)$}\label{sec:BL2}
Let $\{X_1,\dots,X_{m_1}\}$ be an orthonormal basis of the horizontal layer $V_1$ of $\mathfrak g$. We denote by $\Delta_\G$ the positive
\emph{horizontal sub-Laplacian} 
$$
\Delta_\G := -\sum_{j=1}^{m_1}X_j^2.
$$
Let $s\in \N$, $s\ge 1$. If
$X^J$ is a homogeneous differential operator in the horizontal derivatives of degree $d(J)=s$,
  we can consider the distributional derivative of order $s$,  $X^Ju$ of a function $u$.

The vector of horizontal first derivatives is denoted by
$$\nabla_\G u:=(X_1 u,\ldots, X_{m_1}u).$$
If $s\ge 2$ it is convenient to denote by

\begin{equation}\label{Gradient}
|\nabla_\G^s u|:=\left(\sum_{J\,:\,d(J)=s}\,|X^J u|^2\right)^{\frac{1}{2}},
\end{equation}
and by $\|\nabla_{\G}^s u\|_{L^p}$ its $L^p$-norm.

The {horizontal first Beppo Levi space} $BL^{1,p}(\G)$ (also known as the \emph{homogeneous Sobolev space}) is the closure of $\mc D(\G)$ in the $BL^{1,p}$-norm 
\begin{align*}
\|u\|_{BL^{1,p}}= \|\nabla_{\G} u\|_{L^p}.
\end{align*}
By the Sobolev inequality with sharp exponents in Carnot groups (see e.g., \cite{Lu94}) one easily obtain the following result.

\begin{proposition}
    [see Theorem 8.4 in \cite{BFP6}]\label{embedding} 
If $1<p<Q$, a function $u$ belongs to $BL^{1,p}(\G)$ if and only if 
$$
u\in L^{pQ/(Q-p)}(\G)\quad\mbox{and}\quad X_j u\in L^p(\G),\,\forall j=1,\ldots,m_1,
$$
and the two norms are equivalent, i.e.,
$$
\|u\|_{BL^{1,p}(\G)}\approx \|u\|_{L^{pQ/(Q-p)}(\G)} + \sum_{j=1}^{m_1} \|X_j u\|_{L^p(\G)}. 
$$
In particular,
the embedding $BL^{1,p}(\G)\subset L^{pQ/(Q-p)}(\G)$ is continuous.
\end{proposition}

We need to have a notion of Beppo Levi Sobolev space of higher order.
Following Definition $4.9$ in \cite{FR2}, we can give  a very general definition of {homogeneous Beppo Levi-Sobolev spaces} on $\G$, holding for any non negative real number $s$.
\begin{defi}\label{defi.BLG}
Let $s\geq 0$. If $1< p<\infty$, we define the space $BL^{s,p}(\G)$ as the completion of $\mathcal{D}(\G)$ with respect to the norm
$$\|u\|_{BL^{s,p}(\G)}:=\|(\Delta_\G)_p^{s/2}u\|_{L^p(\G)},$$
where $(\Delta_\G)_p^{s/2}$ is the realization of $(\Delta_\G)^{s/2}$ in $L^p(\G)$ defined following \cite{folland}.
The resulting space is a reflexive Banach space.
\end{defi}

\begin{proposition}[Proposition $4.17$ in \cite{FR2}]\label{FR2 zero}
If $K$ is a kernel of type $0$, $1<p<\infty$, $s\geq 0$, then the mapping $\,u\mapsto u\ast K$ defined for $u\in \mathcal{D}(\G)$ extends to a bounded operator on $BL^{s,p}(\G)$.
\end{proposition}


We want to prove a characterization of $BL^{s,p}(\G)$ when $s\in\N$, making explicit the norm $\|\cdot\|_{BL^{s,p}(\G)}$ in terms of horizontal derivatives (i.e. to $|\nabla_\G^s\cdot|$ defined in \eqref{Gradient}).

\begin{prop}\label{BLsinN}
If $s\in\N$, then the map $u\mapsto\|\nabla_{\G}^s  u\|_{L^p(\G)}$
is a norm on $BL^{s,p}(\G)$ equivalent to the $BL^{s,p}(\G)$-norm in Definition \ref{defi.BLG}.
\end{prop}
\begin{proof}
Let $u\in \mathcal{D}(\G)$. Let us prove at first that for every $s\geq 1$
$$\|u\|_{BL^{s,p}(\G)}\quad\approx\quad\sum_{i=1}^{m_1}\|X_i u\|_{BL^{s-1,p}(\G)}.$$
Following Theorem $4.10$ in \cite{folland}, if $u\in \mathcal{D}(\G)$ and $X_i\in V_1$, we have
$$X_i u=(\Delta_\G)^{1/2}u\ast X_i R_1,$$
where $R_1$ is the Riesz kernel associated to the sub-Laplacian $\Delta_\G$ defined in Proposition $3.17$ in \cite{folland}, and $X_i R_1$ is a kernel of type $0$. Thus, by Proposition \ref{FR2 zero}, we have
\begin{align*}
\|X_i u\|_{BL^{s-1,p}(\G)}&=\|(\Delta_\G)^{1/2}u\ast X_i R_1\|_{BL^{s-1,p}(\G)}\\
&\leq c_p \|(\Delta_\G)^{1/2}u\|_{BL^{s-1,p}(\G)} =c_p \|u\|_{BL^{s,p}(\G)}.
\end{align*}
On the other hand, 
$$(\Delta_\G)^{1/2}u=-\sum_{i=1}^{m_1}(X_i u)\ast \Delta_\G(K_i \ast R_1),$$
where $\Delta_\G(K_i \ast R_1)$ is a kernel of type $0$. Again, by Proposition \ref{FR2 zero}, we have
$$\|u\|_{BL^{s,p}(\G)}=\|(\Delta_\G)^{1/2}u\|_{BL^{s-1,p}(\G)}\leq c_p \sum_{i=1}^{m_1}\|X_i u\|_{BL^{s-1,p}(\G)}.$$
We have completed the proof since
$$\|u\|_{BL^{s,p}(\G)}\approx \sum_{i=1}^{m_1}\|X_i u\|_{BL^{s-1,p}(\G)} \approx\cdots\approx\sum_{J\,:\,d(J)=s}\|X^J u\|_{L^p(\G)}.$$ 
\end{proof}

By Proposition \ref{BLsinN}, from now on, if $s\in\N$,  $BL^{s,p}(\G)$ can be seen as the closure of $\mathcal{D}(\G)$ with respect to the norm
$$
\| u\|_{\BL{s}{p}} := \|\nabla_\G^s u\|_{L^p(\G)}.
$$
As a consequence of Hardy-Littlewood-Sobolev inequalities for Beppo Levi spaces proved in Theorem $4.24$ in \cite{FR2}, we have the following characterization that generalizes Proposition \ref{embedding} (for an explicit computation, we also refer  to Corollary $1.4.6$ in \cite{Imperato08}).
\begin{teo}\label{car_spazio} If $s\in\N, s\geq 1,\,1<p<Q$ and $ps<Q$, we have that
$$u \in \,\BL{s}{p} \quad\mbox{if and only if}\quad \nabla_{\G}^ju \in L^{\frac{pQ}{Q-(s-j)p}}({\mathbb G}),\,\,\text{for $j=0,\cdots,s$}$$
and the two norms are equivalent, i.e.
\begin{equation}\label{normaLu}
\|u\|_{\BL{s}{p}}\,\,\approx\,\, \|u\|_{L^{pQ/(Q-sp)}(\G)} + \sum_{j=1}^{s} \|\nabla_\G^j u\|_{L^{\frac{pQ}{Q-(s-j)p}}({\mathbb G})}.
\end{equation}
In particular, the embedding $\BL{s}{p}\hookrightarrow L^{pQ/(Q-sp)}(\G)$ is continuous.
\end{teo}

\begin{oss}\label{convweak}
If $P(X)$ is a homogeneous left-invariant differential operator of order $m$ with respect to horizontal derivatives, then 
$$P(X):BL^{s,p}(\G)\to BL^{s-m,p}(\G)$$
is continuous, for every $s\geq m$. In particular, if $(u_k)_{k\in\N}\subseteq BL^{s,p}(\G)$ is such that $u_k\rightharpoonup u\in BL^{s,p}(\G)$ in $BL^{s,p}(\G)$, then $P(X)u_k\rightharpoonup P(X)u$ in $BL^{s-m,p}(\G)$ by continuity.
\end{oss}

\medskip

\subsection{Multilinear algebra in Carnot groups}\label{sec:rumin}

We denote $\Lgiu{0}{\g}=\R=\Lsu{0}{\g}$ and by $\Lsu{1}{\g}$ the dual space of $\g=:\Lgiu{1}{\g}$. Let $\{X_1,\ldots,X_N\}$ be an orthonormal basis of $\g$ and let $\{\theta_1,\ldots,\theta_N\}\subseteq\Lsu{1}{\g}$ be the dual basis. Now, following \cite{federer}, we can define the exterior algebras $\Lgiu{\bullet}{\g}$ and $\Lsu{\bullet}{\g}$ of $\g$ and $\Lsu{1}{\g}$, respectively, as
$\Lgiu{\bullet}{\g}=\bigoplus_{h=0}^N \Lgiu{h}{\g}\,\text{and}\,\Lsu{\bullet}{\g}=\bigoplus_{h=0}^N\Lsu{h}{\g},$
where
\begin{align}
\Lgiu{h}{\g}&=\Span\{X_{i_1}\wedge\cdots\wedge X_{i_h}\,|\,1\leq i_1<\cdots<i_h\leq N\}, \nonumber \\
\Lsu{h}{\g}&=\Span\{\theta_{i_1}\wedge\cdots\wedge \theta_{i_h}\,|\,1\leq i_1<\cdots<i_h\leq N\}. \nonumber
\end{align}
We call $\Lgiu{h}{\g}$ and $\Lsu{h}{\g}$ the spaces of \textit{$h$-vectors} and \textit{$h$-covectors}, respectively. We denote by $\Theta^h$ the basis $\{\theta_{i_1}\wedge\cdots\wedge \theta_{i_h}\,|\,1\leq i_1<\cdots<i_h\leq N\}$ of $\Lsu{h}{\g}$, for every $h=0,\ldots,N$.
We extend it canonically to $\Lsu{h}{\g}$ such that $\Theta^h$ is orthonormal.

Chosen the above basis as positive one, we can define the \textit{$\star$-Hodge isomorphism} $\star:\Lsu{h}{\g}\stackrel{\sim}{\longrightarrow}\Lsu{N-h}{\g}$ associated to the scalar product $\langle\cdot,\cdot\rangle$ and the volume form $\theta_1\wedge\cdots\wedge\theta_N$.
We recall that $\star\star\varphi=(-1)^{h(N-h)}\varphi$, and $\varphi\wedge\star \psi=\langle \varphi,\psi\rangle\,\theta_1\wedge\cdots \theta_N$, for every $\varphi,\psi\in\Lsu{h}{\g}$.
\medskip

We recall that an $h$-form $\alpha$ on $\G$ is said to be \textit{left-invariant} if $\tau_g^\#\alpha=\alpha$, for every $g\in\G$, where $\tau_g^\#$ is the pullback of $\tau_g$. We recall also the notion of \textit{weight of a form}, which will bring to a spitting of the de Rham differential that reflects the stratification of the Carnot group. 
If $\alpha\neq 0, \alpha\in\Lsu{1}{\g}$, we say that $\alpha$ has \textit{pure weight $p$} if 
$\alpha^\natural\in V_p$. In this case, we set $\omega(\alpha)=p$.
More generally, if $\alpha\in\Lsu{h}{\g},\alpha\neq 0, h\neq 0$, we say that $\alpha$ has pure weight $p$ if $\alpha$ is a linear combination of covectors $\theta_{i_1}\wedge\cdots\wedge\theta_{i_h}$ with $\omega(\theta_{i_1})+\cdots+\omega(\theta_{i_h})=p$. The volume form has weight $Q$.
By \cite{BFTT}, Remark 2.4,  if $\alpha,\beta\in\Lsu{h}{\g}$ have different pure weights, then they are orthogonal with respect to $\langle\cdot,\cdot\rangle$. Thus, the covector space can be split as
$\Lsu{h}{\g}=\bigoplus_{p=M_h^{min}}^{M_h^{max}}\Lsu{h,p}{\g},$
where $\Lsu{h,p}{\g}$ is the space of linear combinations of $h$-covectors of pure weight $p$, and $M_h^{min}, M_h^{max}$ are respectively the smallest and the largest weight of $h$-covectors. Since the elements of the basis $\Theta^h$ have pure weights, a basis of $\bigwedge\nolimits^{h,\,p}\g$ is given by $\Theta^{h,\,p}:=\Theta^h\cap\bigwedge\nolimits^{h,\,p}\g.$ Moreover, we have the decomposition of $\Omega^{h}(\G)=:\Omega^{h}$ with respect to the weights of the forms:
\begin{equation*}
\Omega^h=\bigoplus_{p=M_h^{min}}^{M_h^{max}}\Omega^{h,p},
\end{equation*}
where $\Omega^{h,p}$ is the space of $h$-forms of weight $p$, i.e.\,the set of the smooth sections of the subbundle $\bigwedge\nolimits^{h,\,p}\g$.

If $\alpha\in\Omega^{h,p}$ is a left-invariant $h$-form of pure weight $p$ such that $d\alpha\neq 0$, then $\omega(\alpha)=\omega(d\alpha)$, i.e.\,$d\left(\Lsu{h,p}{\g}\right)\subseteq\Lsu{h+1,p}{\g}$ (see \cite{BFTT}).  Thus, if $\alpha$ is a general $h$-form of pure weight $p$, it can be expressed as $\alpha=\sum_{\theta^h_i\in\Theta^{h,p}} f_i\,\theta_i^h.$
Then, we can write
$$d\alpha=d_0\alpha+d_1\alpha+\cdots+d_\kappa\alpha$$
where $d_0$ does not increase the weight of the forms, $d_1$ increases the weight of one and so on. Precisely,
\begin{gather}\label{eq.splitdiff}
\begin{split}
&d_0\alpha=\sum_{\theta^h_i\in\Theta^{h,p}} f_i\,d\theta_i^h\,\,\in\Omega^{h+1,\,p}, \\
&d_\ell\alpha=\sum_{\theta^h_i\in\Theta^{h,p}}\,\sum_{X_j\in V_\ell}\,X_j(f_i)\,\theta_j\wedge\theta_i^h\,\,\in\Omega^{h+1,\,p+\ell},\quad\forall\,\ell\in\{1,\ldots \kappa\},
\end{split}
\end{gather}
In particular, $d_0$ is an algebraic operator,  in the sense that its action can be identified at any point
with the action of an operator from $\Lsu{h}{\g}$ to $\Lsu{h+1}{\g}$, that we denote again by $d_0$, through the
formula $(d_0\alpha)(x)=\sum_{\theta^h_i\in\Theta^{p,h}} f_i(x)\,d_0\theta_i^h$.

One can define the $L^2$-inner product $\Scal{\alpha}{\beta}$ of two smooth differential forms $\alpha$ and $\beta$, with at least one of them having compact support, as
\begin{align}\label{scalL2}
\Scal{\alpha}{\beta}=\int_\G \alpha\wedge\star\beta\int_\G \langle \alpha, \beta\rangle \, dV.
\end{align}

We also introduce the operator $\delta_0$ as the $L^2$-adjoint operator of $d_0$ in $\Omega^\bullet$; it is again an algebraic operator
preserving the weight.

\subsection{Rumin's complex in Carnot groups}
Now, we recall briefly the definition of Rumin's complex due to M.\,Rumin (\cite{rumin_jdg}, \cite{rumin_cras}). For a detailed presentation see \cite{BFTT}. An alternative construction of the complex can be found in \cite{fischer_tripaldi}.

\begin{defi}\label{defi.intrforms} 
If $0\leq h\leq N$, we set
\begin{equation}\label{eq.E0}
E_0^h:=\ker\,d_0\cap\ker\,\delta_0=\ker\,d_0\cap\left(\Im\,d_0\right)^\bot\,\, \subseteq\,\Omega^h.
\end{equation}
\end{defi}
We set $N_h:=\dim E_0^h$. We observe that the construction of $E_0^h$ is left-invariant, thus we can see it as a subbundle of $\Lsu{h}{\g}$, generated by left translations and still denoted by $E_0^h$. We can write
\begin{equation}\label{eq.splitE0hbyweight}
E_0^h=\bigoplus_{p=M_h^{min}}^{M_h^{max}}E_0^{h,p},
\end{equation}
where $E_0^{h,p}:=E_0^h\,\cap\,\Omega^{h,p}$. In particular, $E_0^h$ and $E_0^{h,\,p}$ inherit the scalar product on the fibers and we obtain a left-invariant orthonormal basis $\Xi_0^h=\{\xi_i^h\}_i$ of $E_0^h$ such that, if we set $\Xi_0^{h,\,p}:=\Xi_0^{h}\cap\Omega^{h,\,p}$ a left-invariant orthonormal basis of $E_0^{h,\,p}$, we have a  basis of $E_0^h$ given by
$$\Xi_0^h=\bigcup_{p=M_h^{min}}^{M_h^{max}}\Xi_0^{h,\,p}.$$
The set of indices $\{1,\ldots,N_h\}$ can be written as the union of finite sets (possibly empty) of indices, i.e.\,$\{1,\ldots,N_h\}=\bigcup_{p=M_h^{min}}^{M_h^{max}} I_{0,p}^h,$
where $j\in I_{0,p}^h$ if and only if $\xi_j^h\in\Xi_0^{h,p}$.

Moreover, we have also the Hodge duality between Rumin forms as $\star E_0^h = E_0^{N-h},$ for every $h=0,\ldots,N$.

The following  theorem summarizes the construction of 
the intrinsic differential $d_c$ (for details, see \cite{rumin_grenoble}
and \cite{BFTT}, Section 2) .

\begin{teo}\label{teo.splitdeRhamanddC}
The de Rham complex $(\Omega^\bullet,d)$ splits in the direct sum of two subcomplexes, denoted by Rumin as $(E^\bullet,d)$ and $(F^\bullet,d)$. Moreover, we have the following: let us denote by $\Pi_E$ and $\Pi_{E_0}$ the projections from $\Omega^\bullet$ on $E^\bullet$ and $E_0^\bullet$, respectively, and let us set
$$d_c:=\Pi_{E_0}d\,\Pi_E : E_0^h\to E_0^{h+1}\quad\text{for every $h=0,\ldots,N-1$}.$$
Then $(E_0^\bullet, d_c)$ is a complex homotopically equivalent to the de Rham's one, called Rumin's complex, such that
\begin{itemize}
\item The differential $d_c$ acting on $h$-forms can be identified, with respect to the bases $\Xi_0^h$ and $\Xi_0^{h+1}$, with a matrix-valued differential operator $L^h:=(L^h_{i j})$. If $j\in I_{0,p}^h$ and $i\in I_{0,q}^{h+1}$, then the $L^h_{i j}$'s are homogeneous left-invariant differential operators of order $q-p\geq 1$ in the horizontal derivatives, and $L^h_{i j}=0$ if $j\in I_{0,p}^h$ and $i\in I_{0,q}^{h+1}$, with $q-p<1$;
\item Denoting by $\delta_c$ the $L^2$-formal adjoint of $d_c$ on $E_0^h$, we have $$\delta_c=(-1)^{N(h+1)+1}\star d_c\star.$$
\end{itemize}
\end{teo}
The last formula can be seen as a consequence of Stokes formula (see Remark 2.16 in \cite{BFTT}). 
If $\alpha$ has degree $(h-1)$ and $\beta$ has degree $h$, recalling that $d_c\star\beta$ is a $N-h+1$ form and that on a $(N-h+1)$-form $\star\star=(-1)^{(h-1)(N-h+1)}$, we get
\begin{align*}
\Scal{d_c\alpha}{\beta}&=\int d_c\alpha\wedge\star\beta=-\int (-1)^{(h-1)}\alpha\wedge d_c\star\beta
\\&=(-1)^{h}(-1)^{(h-1)(N-h+1)}\Scal{\alpha}{\star d_c\star \beta}\\&=(-1)^{N(h+1)+1}\Scal{\alpha}{\star d_c\star \beta} =:\Scal{\alpha}{\delta_c\beta}.
\end{align*}

By  H\"older's inequality,  \eqref{scalL2} extends  to the dual pairing of a $h$-form 
$\alpha$  in $L^p$ with a $h$-form $\beta$ in $L^{p'}$, 
\[
\Scal{\alpha}{\beta} =\int_\G \alpha\wedge\star\beta=\int_\G \langle \alpha, \beta\rangle \, dV,
\]
where \(1 \leq p \leq \infty\) and \(\frac{1}{p} + \frac{1}{p'} = 1\).

\medskip
For every $1\leq p< \infty$ and $s\in\N$, we denote by $W^{s,p}(\G,E_0^h), \mathcal{D}(\G,E_0^h)$, $C^\infty(\G,E_0^h)$, $BL^{s,p}(\G,E_0^h)$ the spaces of all the sections of $E_0^h$ such that their components with respect to $\Xi_0^h$ belong to the corresponding scalar spaces. Once the basis $\Theta^h$ is chosen (thus the basis $\Xi_0^h$), the above spaces can be identified with $\left(W^{s,p}(\G)\right)^{N_h}$,$ \left(\mathcal{D}(\G)\right)^{N_h}$,$ \left(C^\infty(\G)\right)^{N_h}$, $\left(BL^{s,p}(\G)\right)^{N_h}$. 

In the sequel, we denote the distribution duality between a $h$-distribution $\alpha\in\mathcal{D}'(\G,E_0^h)$ and a $h$-form $\beta\in\mathcal{D}(\G,E_0^h)$ as $\Scal{\alpha}{\beta}_{\mc D', \mc D}$, which, for simplicity, will be denoted again by $\Scal{\alpha}{\beta}$ if clear from the context.

\medskip

\section{The Cartan group}\label{ex.Cartanesempio}

From now on, we denote by $\G$ the step-3 free Carnot group, also known as the \textit{Cartan group}. More precisely, we consider  the free group of step $3$ with $2$ generators, whose Lie algebra is 
$$\g=V_1\oplus V_2\oplus V_3,$$
where $V_1=\Span\{X_1,X_2\}$, $V_2=\Span\{X_3\}$ and $V_3=\Span\{X_4,X_5\}$ with the only nontrivial commutation rules
$$[X_1,X_2]=X_3,\quad [X_1,X_3]=X_4,\quad [X_2,X_3]=X_5.$$
A basis of $\g$ of left-invariant vector fields is given by
\begin{gather*}
\begin{split}
X_1&=\de_{\dsy x_1},\\
X_2&=\de_{\dsy x_2}+x_1\de_{\dsy x_3}+\frac{x_1^2}{2}\de_{\dsy x_4}+x_1 x_2\de_{\dsy x_5}, \\
X_3&=\de_{\dsy x_3}+x_1\de_{\dsy x_4}+x_2\de_{\dsy x_5}, \\
X_4&=\de_{\dsy x_4}, \\
X_5&=\de_{\dsy x_5}.
\end{split}
\end{gather*}
We can identify $\G$ with $\R^5$.
The stratification of $\g$ induces a family of nonisotropic dilations $\{\delta_\lambda\}_{\lambda>0}$ in $\G$ so that, for any $x=(x_1,\ldots,x_5)\in\G$,
$$\delta_\lambda(x)=(\lambda x_1,\lambda x_2,\lambda^2 x_3, \lambda^3 x_4, \lambda^3 x_5).$$

\bigskip

\subsection{Rumin's complex in the Cartan group}\label{ex.Cartanrumin}
An explicit computation of the Rumin complex in this group already appeared in \cite{BFTT}, Example B.7, and, following the notations therein,  we can  denote by $\theta_1,\cdots, \theta_5$ the dual left-invariant forms of $X_1,\cdots,X_5$. More recently, this group has been treated in \cite{BT}, and we refer the interested reader to it for more details (see Section 4 in \cite{BT}). The Rumin complex works perfectly on the  Cartan group, since Rumin forms have only one weight in every fixed degree. As noticed in the previous section, this feature is not a general property in Carnot groups (see e.g. in \cite {BFTr}, Section 3 and 4, two explicit examples given by the Engel group and the 7-dimensional quaternionic Heisenberg group).  Even in  free Carnot groups this phenomenous do not occur. Indeed, in general, in a free Carnot group of step $\kappa$  it is only true that forms in $E_0^1$ have pure weight 1 and forms in $E_0^2$ have pure weight $\kappa+1$ (see Theorem 5.9 in \cite{FT4}), but  the weight might change for forms of other degrees (see Remark 3.11 in \cite{BF4}.

We now consider an orthonormal bases of Rumin's forms (see Section $4$ in \cite{BT} for its explicit expression), Hence, according to Theorem \ref{teo.splitdeRhamanddC},  $d_c$ can be seen as a matrix valued differential operator as follows:

\begin{itemize}
\item \text{$d_c: E_0^0\to E_0^1$ is given by}
$$d_c=\begin{pmatrix}
X_1 \\
X_2
\end{pmatrix}$$
\item \text{$d_c: E_0^1\to E_0^2$, }
$$d_c=\begin{pmatrix}
-X_1^2X_2-X_1X_3-X_4 & X_1^3 \\
-\sqrt{2}(X_1X_2^2+X_5) & \sqrt{2}(X_2X_1^2-X_4)\\
-X_2^3 & X_2^2X_1-X_2X_3-X_5
\end{pmatrix}$$
\item \text{$d_c: E_0^2\to E_0^3$}
$$d_c=\begin{pmatrix}
-X_1X_2-X_3 & \frac{1}{\sqrt{2}}X_1^2 & 0\\
-\frac{1}{\sqrt{2}}X_2^2 & -\frac{3}{2}X_3 & \frac{1}{\sqrt{2}}X_1^2 \\
0 & -\frac{1}{\sqrt{2}}X_2^2 & X_2X_1-X_3
\end{pmatrix}$$
\item \text{$d_c: E_0^3\to E_0^4$,}
$$d_c=\begin{pmatrix}
(X_1X_2+X_3)X_2-X_5 & \sqrt{2}(-X_1^2X_2+X_4) & X_1^3 \\
X_2^3 & -\sqrt{2}(X_2^2X_1+X_5) & X_2X_1^2-X_3X_1+X_4
\end{pmatrix}$$
\item \text{$d_c: E_0^4\to E_0^5$,}
$$d_c=\begin{pmatrix}
-X_2 & X_1
\end{pmatrix}$$
\end{itemize}

\medskip

We can notice that $d_c$ is always a homogeneous differential operator, but its order changes, depending on the degree $h$ of the form it acts on. More precisely, it is a higher order operator  when acting on  forms $h=1,2,3$: it has order $3$  if $h=1,3$, whereas it has order $2$ when $h=2$.

By Hodge duality, one can find the explicit expressions for the codifferentials $\delta_c$. Once expressed in terms of the same ordered bases, the matrix form of $\delta_c$ can be expressed as the transpose of the matrix representing the differentials $d_c$ (up to a sign). Hence, as differential operators, $\delta_c$ has order one when acting on $0$ or $5$ forms, it has order $3$ when acting on $2$ and $4$ forms, and has order $2$ on $E_0^3$ (in \cite{BT}, Section 4,  there is  the explicit expression of $\delta_c$ in any degree $h$).

\subsection{Leibniz formula}

When $d_c$ is not a first  order differential operator, $(E_0^\bullet,d_c)$ stops behaving like a differential module. This is the source of many complications. For example, if we want to localize our estimates by means of cut-off functions, we should need a Leibniz formula. The classical Leibniz formula for the de Rham complex $d(\alpha\wedge\beta)=d\alpha\wedge\beta\pm \alpha\wedge d\beta$ in general fails to hold 
 also in the simpler case of Heisenberg groups, where $d_c$, as a differential operator, has only two possible orders $1$ and $2$ (see \cite{BFT2}-Proposition A.7).   It turns out that, also in the simplest example of Carnot group, the first Heisenberg group $\H^1$, the  Leibniz formula for $d_c$ has a complicated expression (see e.g. \cite{BFP2}). In this paper, we need to analyze the expression of the Leibniz formula in the case of the Cartan group. We start with a general result about commutations.
\begin{lemma}\label{commutP}
Let $\zeta\in C^\infty(\G,\R)$ and let $\mathcal{P}:C^\infty(\G,E_0^{h_1})\to C^\infty(\G,E_0^{h_2})$ be a left-invariant differential operator of order $m$ with respect to horizontal derivatives, for some $0\leq h_1,h_2\leq 5$. Then $[\mathcal{P},\zeta]: C^\infty(\G,E_0^{h_1})\to C^\infty(\G,E_0^{h_2})$ is a differential operator of order $m-1$ that can be expressed as
$$[\mathcal{P},\zeta]=\sum_{j=0}^{m-1}\,P_{m-j}^{h_1}(X^{j+1}\zeta),$$
where $P_{m-j}^{h_1}(X^{j+1}\zeta):C^\infty(\G,E_0^{h_1})\to C^\infty(\G,E_0^{h_2})$ is a linear homogeneous differential operator of order $m-j$ with coefficients depending
only on the horizontal derivatives of order $j+1$ of $\zeta$.
\end{lemma}
\begin{proof}
Let $\{\xi^{h_1}_j\}_{j=1,\ldots,N_{h_1}}$ (resp.\,$\{\xi^{h_2}_j\}_{j=1,\ldots,N_{h_2}}$) be the left-invariant basis of $E_0^{h_1}$ (resp.\,$E_0^{h_2}$). Thus, we can see $\mathcal{P}$ as a matrix-valued operator as
$$\mathcal{P}=(\mathcal{P}^{ij})_{i,j}:(C^\infty(\G))^{N_{h_1}}\to (C^\infty(\G))^{N_{h_2}}.$$
Fixed $i\in\{1,\ldots,N_{h_1}\}, j\in\{1,\ldots,N_{h_2}\}$, let us define $P:=\mathcal{P}^{ij}$. Since $\mathcal{P}$ is a left-invariant differential operator of order $m$, by Poincaré-Birkoff-Witt theorem, $P$ can be expressed as
$$P=\sum_{d(J)\leq m}\,c_J X^J$$
for some constant coefficients $c_J\in \R$.  If $u\in C^\infty(\G)$, using the notation in \eqref{derivate}, we have
\begin{align*}
[P,\zeta]\alpha&=P(\zeta\alpha)-\zeta P\alpha= \sum_{d(J)\leq m}c_JX_1^{j_1}\cdots X_N^{j_N}(\zeta\,u)-\zeta Pu\\
&=\sum_{d(J)\leq m}c_JX_1^{j_1}\cdots X_N^{j_N-1}((X_N\zeta)u+\zeta(X_N u))-\zeta Pu\\
&=\sum_{d(J)\leq m}c_JX_1^{j_1}\cdots X_N^{j_N-2}((X^2_N\zeta)u+X_N\zeta X_N u+\zeta(X^2_N u))-\zeta Pu\\
&\,\,\,\vdots\\
&=\sum_{s=0}^{m-1} \hat P_{m-s}(X^{s+1}\zeta)u +\zeta Pu -\zeta Pu=\sum_{s=0}^{m-1} \hat P_{m-s}(X^{s+1}\zeta)u,
\end{align*}
where $\hat P_{m-s}(X^{s+1}\zeta)$ are horizontal homogeneous differential operators of order $m-s$, with coefficients depending only on the horizontal derivatives of order $s+1$ of $\zeta$. Putting together the computations for every $i,j$, we complete the proof.
\end{proof}
Using the previous Lemma we can obtain a result related to the Leibniz formula in the Cartan group. Depending on the degree of the differential forms involved the formula has different expressions.

\begin{coro}
\label{leibniz} If $\zeta$ is a smooth real function, then

i)
\begin{itemize}\item If $h=0,4$,  then on $E_0^h$ we have:
$$
[d_c,\zeta] = P_0^h(X\zeta),
$$
where $P_0^h(X\zeta): E_0^h \to E_0^{h+1}$ is a linear homogeneous differential operator of order zero with coefficients depending
only on the horizontal derivatives of $\zeta$. If $h=1,5$, an analogous statement holds if we replace
$d_c$ by $\delta_c$. More precisely, if $h=1,5$ thus on  $E_0^h$ we have:
$$
[\delta_c,\zeta] = P_0^h(X\zeta),
$$
where $P_0^h(X\zeta): E_0^h \to E_0^{h-1}$ is a linear homogeneous differential operator of order zero with coefficients depending
only on the horizontal derivatives of $\zeta$.
\item If $h= 1,3$, then on $E_0^h$ we have
$$
[d_c,\zeta] = P_2^h(X\zeta)+P_1^h(X^2\zeta) + P_0^h(X^3\zeta) ,
$$
where for $j=0,1,2$ the $P_j^h(X^{3-j}\zeta):E_0^{h} \to E_0^{h+1}$ are linear homogeneous differential operators of order j, and therefore
horizontal,  with coefficients depending
only on the horizontal derivatives of order $3-j$ of $\zeta$. An analogous statement holds if $h=2,4$ and  we replace $d_c$ by $\delta_c$ keeping in mind that the $P_j^h(X^{3-j}\zeta):E_0^{h} \to E_0^{h-1}$.

\item If $h=2$ then on $E_0^2$ we have:
$$
[d_c,\zeta] = P_1^2(X\zeta) + P_0^2(X^2\zeta),
$$
where $P_1^2(X\zeta):E_0^2 \to E_0^{3}$ is a linear homogeneous differential operator of order 1 (and therefore
horizontal) with coefficients depending
only on the horizontal derivatives of $\zeta$, and where $P_0^2(X^2\zeta): E_0^2 \to E_0^{3}$ is a linear homogeneous differential operator in
the horizontal derivatives of order 0,
 with coefficients depending
only on second order horizontal derivatives of $\zeta$. An analogous statement holds if $h=3$ and we replace
$d_c$ by $\delta_c$ and $P_j^h:E_0^{3} \to E_0^{2}$, for $j=0,1$.
\end{itemize}

ii)
\begin{itemize}
\item If $h=1,5$, then
$$
[d_c\delta_c,\zeta] = P_1^h(X\zeta) + P_0^h(X^2\zeta) ,
$$
where $P_1^h(X\zeta):E_0^h \to E_0^{h}$ is a linear homogeneous differential operator of order 1, and therefore
horizontal, with coefficients depending
only on the horizontal derivatives of $\zeta$, and where $P_0^h(X^2\zeta): E_0^h \to E_0^{h}$ is a linear homogeneous differential operator in
the horizontal derivatives of order 0
 with coefficients depending
only on second order horizontal derivatives of $\zeta$.
\item If $h=2,4$, then
$$
[d_c\delta_c,\zeta] = \sum_{j=0}^5 P_j^h(X^{6-j}\zeta)  ,
$$
where for $j=0,1,\cdots,5$, the $P_j^{h}(X^{6-j}\zeta):E_0^{h} \to E_0^{h}$ are linear homogeneous differential operators of order j (and therefore
horizontal) with coefficients depending
only on the horizontal derivatives of order $6-j$ of $\zeta$.
\item If $h= 3$, then
$$
[d_c\delta_c,\zeta] = P_3^{3}(X\zeta) +P_2^{3}(X^2\zeta)  +P_1^{3}(X^3\zeta)  + P_0^{3}(X^4\zeta) ,
$$
where for $j=0,1,2,3$, the $P_j^{3}(X^{4-j}\zeta):E_0^{3} \to E_0^{3}$ are linear homogeneous differential operators of order j and therefore
horizontal) with coefficients depending
only on the horizontal derivatives of order $4-j$ of $\zeta$.
\end{itemize}
\end{coro}


\medskip
\subsection{Laplacians in the Cartan group}\label{sec:GNineqestteo}

Throughout this paper, we consider the following  homogeneous left-invariant Laplacians, already considered in \cite{BT} (see Definition 4.3 therein):
\begin{equation}\label{eq.defLaplCartanGG}
\Delta_{R,h}=\left\{
\begin{array}{ll}
\delta_cd_c,\quad &\text{if $h=0$},\\
(d_c\delta_c)^3+\delta_cd_c,\quad &\text{if $h=1$},\\
(d_c\delta_c)^2+(\delta_cd_c)^3,\quad &\text{if $h=2$},\\
(d_c\delta_c)^3+(\delta_cd_c)^2,\quad &\text{if $h=3$},\\
d_c\delta_c+(\delta_cd_c)^3,\quad &\text{if $h=4$},\\
d_c\delta_c,\quad &\text{if $h=5$}.
\end{array}\right.
\end{equation}
Notice that $\Delta_{R,0}=-(X_1^2+X_2^2)$ is the usual sub-Laplacian on $\G$. 

Once a basis of $E_0^h$ is fixed, we can identify the operators $\Delta_{R,h}$ with a matrix-valued map, still denoted by $\Delta_{R,h}$ as
$$\Delta_{R,h}=(\Delta_{R,h}^{i j})_{i,j=1,\ldots,N_h}:\mathcal{D}'(\G,\R^{N_h})\to \mathcal{D}'(\G,\R^{N_h}).$$
If 
$$M_h\ \text{denotes the order of $\Delta_{R,h}$}
$$
  as a differential operator, we have  $M_0=M_5=2$, $M_1=M_4=6$ and $M_2=M_3=12$.
 When $M_h\ge Q$,  we will often avoid to deal directly with  $\Delta_{R,h}^{-1}$ and instead work directly with some of its derivatives that are homogeneous kernels of type $<Q$ (see Theorem \ref{teo.kernelCartan23} below). 

  \medskip
  
In \cite{DH} it was proved that $\Delta_{R,h}$ satisfy the Rockland condition, which implies the hypoellipticity of the operators $\Delta_{R,h}$. 

As an immediate consequence of the fact that $d_c^2=0$ and $\delta_c^2=0$ we
have the following lemma. The proof of this result is omitted since it requires only minor modification of Lemma 4.11  of \cite{BFP2}, following by the very definition of $\Delta_{R,h}$.
\begin{lemma}\label{comm} 
If $\alpha\in \mathcal{D}(\G, E_0^h)$, then, for $h=0,\cdots,5$,
\begin{itemize}
\item[i)]$
(d_c\delta_c)^2d_c\Delta_{R, 0}\alpha = \Delta_{R, 1} d_c\alpha,\quad\quad\quad\delta_c\Delta_{R, 1}\alpha = \Delta_{R, 0} (\delta_cd_c)^2\delta_c\alpha$, 

\item[ii)]$
d_c\delta_cd_c\Delta_{R, 1}\alpha = \Delta_{R, 2} d_c\alpha,\quad\quad\quad \delta_c\Delta_{R, 2}\alpha = \Delta_{R, 1} \delta_cd_c\delta_c\alpha$,  

\item[iii)] $d_c\Delta_{R, 2}\alpha = \Delta_{R, 3} d_c\alpha,\quad\quad\quad \delta_c\Delta_{R, 3}\alpha = \Delta_{R, 2} \delta_c\alpha$, 

\item[iv)]$d_c\Delta_{R, 3}\alpha = \Delta_{R, 4} d_c\delta_cd_c\alpha,\quad\quad\quad \delta_cd_c\delta_c\Delta_{R, 4}\alpha = \Delta_{R, 3} \delta_c\alpha$,  

\item[v)]$
d_c\Delta_{R, 4}\alpha = \Delta_{R, 5} (d_c\delta_c)^2d_c\alpha,\quad\quad\quad (\delta_cd_c)^2\delta_c\Delta_{R, 5}\alpha = \Delta_{R, 4} \delta_c\alpha$.
\end{itemize}
Moreover, we have also
\begin{equation}\label{eq.commh=2}
d_c\delta_c\Delta_{R,h}\alpha=\Delta_{R,h}d_c\delta_c\alpha, \quad \delta_cd_c\Delta_{R,h}\alpha=\Delta_{R,h}\delta_cd_c\alpha.
\end{equation}
\end{lemma}

\medskip

Recalling that $N=\dim \G=5$ and the homogeneous dimension of $\G$ is $Q=10$, we observe that if $h\neq 2,3$, we have $M_h<Q$. Thus, by \cite{BFT3}, Theorem 3.1,  we can state the following result.

\begin{teo}[see \cite{BFT3}, Theorem 3.1]\label{teo.inverselaplHn}
For every $h\in\{0,1,4,5\}$, there exist $K\in\mathcal{D}'(\G,\R^{N_h})\cap C^\infty(\G\setminus\{0\},\R^{N_h})$
with the following properties:
\begin{enumerate}
\item the kernels $K_{i j}$ are of type $2$ if $h=0,5$ and of type $6$ if $h=1,4$ in the sense of \cite{folland} (i.e.\,they are homogeneous of degree $M_h-Q$, and hence belonging to $L^1_{loc}(\G)$, by Corollary $1.7$ of \cite{folland}). 
\item If $\alpha=(\alpha_1,\ldots,\alpha_{N_h})\in \mathcal{D}(\G,\R^{N_h})$ and if we set $\Delta_{R,h}^{-1}\alpha=\alpha\ast K$ as a shorthand notation of 
$$\Delta_{R,h}^{-1}\alpha:=\left(\sum_j\alpha_j\ast K_{1 j},\ldots,\sum_j\alpha_j\ast K_{N_h j}\right),$$
then $\Delta_{R,h} \Delta_{R,h}^{-1}\alpha=\alpha$ and $\Delta_{R,h}^{-1}\Delta_{R,h}\alpha=\alpha$. Moreover,  $\Delta_{R,h}^{-1}=\ccheck \Delta_{R,h}^{-1}$.
\end{enumerate}
\end{teo}

When $M_2=M_3=12>Q$ the Laplacians are homogeneous hypoelliptic operators,  and the associated fundamental solutions  belong to $\mathcal{S}'(\G,E_0^h)$ and are of the form 
$$\widetilde{\mathcal{K}}=\mathcal{K}_0+p(x)\log(\rho(x)),$$
where $\mathcal{K}_0$ is a vector-valued function $C^\infty(\G\setminus\{0\})$, homogeneous of degree $12-Q$, and $p(x)$ is a homogeneous vector-valued polynomial of degree $12-Q$ (see, e.g., Theorem 3.2.40-(b) in \cite{FR} and also Theorem 3.1 in \cite{VSY22}). 

Recently, Theorem 3.1 in \cite{BFT3} has been extended to cover the case of operators of order $>Q$. In our setting, considering the operators $\Delta_{R,2}$ and $\Delta_{R,2}$, the results of \cite{VSY22} reads as follows.

\begin{teo}[see \cite{VSY22}, Theorem $3.3$]\label{teo.kernelCartan23}
For every $h\in\{2,3\}$, there exists $K\in\mathcal{D}'(\G,\R^{N_h})$ such that, for every $\alpha=(\alpha_1,\ldots,\alpha_{N_h})\in \mathcal{D}(\G,\R^{N_h})$, if we set $\Delta_{R,h}^{-1}\alpha=\alpha\ast K$ as a shorthand notation of 
$$\Delta_{R,h}^{-1}\alpha:=\left(\sum_j\alpha_j\ast K_{1 j},\ldots,\sum_j\alpha_j\ast K_{N_h j}\right),$$
then it holds that, for any $\alpha\in \mathcal{D}(\G,\R^{N_h})$,
\begin{enumerate}
\item $\dsy\Delta_{R,h} \Delta_{R,h}^{-1}\alpha=\Delta_{R,h}(\alpha\ast K)=\alpha;$
\item 
$\alpha-(\Delta_{R,h}\alpha)\ast  K=p_h,
$ where $p_h(x)$ is a polynomial of $x$ whose non-isotropic degree is at most $2\,(=12-Q)$. We shortly write  $\Delta_{R,h}^{-1} \Delta_{R,h}\alpha=\alpha+p_h$.
\end{enumerate}

Moreover, if $|I|=\ell$, there exist $K_I, \widetilde{K_I}$ vector-valued distributions such that
$$X^I\alpha=\Delta_{R,h}(\alpha\ast K_I)\quad\text{and}\quad X^I\alpha=\Delta_{R,h}\alpha\ast \widetilde{K_I},$$
for every $\alpha\in \mathcal{D}(\G,\R^{N_h})$. If $\ell>2$, then $K_I, \widetilde{K_I}$ are homogeneous vector-valued distributions of degree $2-\ell$, and  if further $\ell<12$, $K_I, \widetilde{K_I}$ are kernels of type $12-\ell$.
\end{teo}

A  Liouville-type result can be obtained exactly as in Proposition $3.2$ of \cite{BFT3} (indeed, in the proof Lemma $3.3$ in \cite{BFT3}, it is enough to choose $M=\max((k-1)Q, 2)$).

\begin{prop}\label{Liouville}
If $\alpha\in \mathcal{S}'(\G,\R^{N_h})$ satisfies $\Delta_{R,h}\alpha=0$, then $\alpha$ is a vector-valued polynomial.
\end{prop}

The following lemma is akin to  Lemma 3.15 in \cite{BFP2}.
\begin{lemma}\label{lemma.commDeltadcdeltac}
If $\alpha\in \mathcal{D}(\G,E_0^h)$, then
$$
\begin{array}{llr}
(i)\quad d_c\Delta_{R,0}^{-1}\alpha=\Delta_{R,1}^{-1}(d_c\delta_c)^2d_c\alpha, & &\text{if $h=0$};\\
(ii)\quad d_c\Delta_{R,1}^{-1}\alpha=\Delta_{R,2}^{-1}d_c\delta_cd_c\alpha, \quad& (\delta_cd_c)^2\delta_c\Delta_{R,1}^{-1}\alpha=\Delta_{R,0}^{-1}\delta_c\alpha, &\text{if $h=1$};\\
(iii)\quad d_c\Delta_{R,2}^{-1}\alpha=\Delta_{R,3}^{-1}d_c\alpha, \quad& \delta_cd_c\delta_c\Delta_{R,2}^{-1}\alpha=\Delta_{R,1}^{-1}\delta_c\alpha, &\text{if $h=2$};\\
(iv)\quad d_c\delta_cd_c\Delta_{R,3}^{-1}\alpha=\Delta_{R,4}^{-1}d_c\alpha, \quad& \delta_c\Delta_{R,3}^{-1}\alpha=\Delta_{R,2}^{-1}\delta_c\alpha, &\text{if $h=3$};\\
(v)\quad (d_c\delta_c)^2d_c\Delta_{R,4}^{-1}\alpha=\Delta_{R,5}^{-1}d_c\alpha, \quad& \delta_c\Delta_{R,4}^{-1}\alpha=\Delta_{R,3}^{-1}\delta_cd_c\delta_c\alpha, &\text{if $h=4$};\\
(vi)\quad \delta_c\Delta_{R,5}^{-1}\alpha=\Delta_{R,4}^{-1}(\delta_cd_c)^2\delta_c\alpha, & &\text{if $h=5$};\\
\end{array}
$$
Moreover, for every $h$,
\begin{equation}\label{eq.commdeltadDelta}
d_c\delta_c \Delta_{R,h}^{-1}\alpha = \Delta_{R,h}^{-1}d_c\delta_c \alpha,\quad\quad \delta_c d_c \Delta_{R,h}^{-1}\alpha = \Delta_{R,h}^{-1}\delta_c d_c\alpha,
\end{equation}
and, if $h=3$,
\begin{equation}\label{commother}
(\delta_cd_c)^2\delta_c\Delta_{R,3}^{-1}\alpha=\Delta_{R,2}^{-1}(\delta_cd_c)^2\delta_c\alpha.
\end{equation}
\end{lemma}
\begin{proof}
 Let us start with $h=1$, proving the formulae involving $\delta_c, d_c$ in (ii). Let $\alpha\in \mathcal{D}(\G,E_0^1)$. We set
\begin{align*}
\tilde{\omega}_1&:=d_c\Delta_{R,1}^{-1}\alpha-\Delta_{R,2}^{-1}d_c\delta_cd_c\alpha,\\
\omega_1&:=(\delta_cd_c)^2\delta_c\Delta_{R,1}^{-1}\alpha-\Delta_{R,0}^{-1}\delta_c\alpha
\end{align*}
Notice that $\Delta_{R,2} \tilde{\omega}_1=\Delta_{R,0}\omega_1=0$. Indeed, by Lemma \ref{comm} (ii) and (i), respectively, we have
\begin{align*}
\Delta_{R,2}\tilde{\omega}_1&=\Delta_{R,2}(d_c\Delta_{R,1}^{-1}\alpha-\Delta_{R,2}^{-1}d_c\delta_cd_c\alpha)\\
&= d_c\delta_cd_c\Delta_{R,1}\Delta_{R,1}^{-1}\alpha-d_c\delta_cd_c\alpha=0.
\end{align*}
\begin{align*}
\Delta_{R,0}\omega_1&=\Delta_{R,0}((\delta_cd_c)^2\delta_c\Delta_{R,1}^{-1}\alpha-\Delta_{R,0}^{-1}\delta_c\alpha)\\
&=\delta_c\Delta_{R,1}\Delta_{R,1}^{-1}\alpha-\delta_c\alpha=0.
\end{align*}
By Proposition \ref{Liouville}, $\tilde{\omega}_1, \omega_1$ are forms with polynomial coefficients.
In addition, by Theorem \ref{teo.inverselaplHn}, Theorem \ref{teo.kernelCartan23} and Proposition \ref{prop.kernelopcheck}, it follows that
$$\tilde{\omega}_1=\alpha\ast \tilde{K}\quad\text{and}\quad\omega_1=\alpha\ast K,$$ 
where $\tilde{K}$ is a kernel of type $3$ and $K$ is a kernel of type $1$. Thus, by Lemma \ref{lemma.estimatekernelpol},  
\begin{equation*}
\tilde{\omega}_1(x)=O(|x|^{3-Q})\quad\text{and}\quad\omega_1(x)=O(|x|^{1-Q})\quad\text{as $x\to\infty$},
\end{equation*}
implying that $\tilde{\omega}_1, \omega_1$ are zero.

The proof for  the other values of $h$ can be obtained analogously, or using the fact that $\star\Delta_{R,h}=\Delta_{R,5-h}\star$. 
Let us just give a gist of the proof \eqref{eq.commdeltadDelta}  when $h=2$, and that of \eqref{commother} if $h=3$. As above, let us set
\begin{align*}
\omega_2&:=\delta_cd_c\delta_c\Delta_{R,2}^{-1}\alpha-\Delta_{R,1}^{-1}\delta_c\alpha, \\
\omega_3&:=(\delta_cd_c)^2\delta_c\Delta_{R,3}^{-1}\alpha-\Delta_{R,2}^{-1}(\delta_cd_c)^2\delta_c\alpha.
\end{align*}
By Theorem \ref{teo.inverselaplHn}, Theorem \ref{teo.kernelCartan23} and Proposition \ref{prop.kernelopcheck}, we have that 
$$\omega_h=\alpha\ast K_h,$$ 
where $K_h$ is a kernel of type $3$ if $h=2$ and of type $2$ if $h=3$. Again, by Lemma \ref{lemma.estimatekernelpol}, we have 
\begin{equation*}
\omega_2(x)=O(|x|^{3-Q})\quad\text{and}\quad\omega_3(x)=O(|x|^{2-Q})\quad\text{as $x\to\infty$}.
\end{equation*}
Now, we can argue as above to conclude.
\end{proof}
\section{$L^p$-Hodge Decomposition and $L^p$-Poincaré inequalities}\label{Lpsection}
In this section we obtain an $L^p$-Hodge decomposition with uniform estimates for $1< p<\infty$, as well as a weak version in the case 
$p=1$. As a consequence of the decomposition for $p>1$, we prove a $(p,q)$-Poincaré  in the spirit of the results in \cite{BFP2} (see Theorem \ref{teo.globalLpPoincineq} below). In contrast, the case $p=1$ is treated separately, since the Poincaré inequality does not follow directly from the weak $L^1$-Hodge decomposition established in  Proposition \ref{L1hodge}.
\medskip

\subsection{$L^p$-Hodge decomposition}\label{sec:LpHodge}

Let $p>1$.
To state the following results concisely we define the exponents
\begin{equation}\label{indici}
s_h:=\left\{
\begin{array}{ll}
1,\quad&\text{if $h=1$}, \\
3,\quad&\text{if $h=2,4$},\\
2,\quad&\text{if $h=3$}
\end{array}\right.\quad\text{and}\quad
r_h:=\left\{
\begin{array}{ll}
3,\quad&\text{if $h=1,3$}, \\
2,\quad&\text{if $h=2$},\\
1,\quad&\text{if $h=4$}.
\end{array}\right.
\end{equation}
We preliminarily prove the Hodge decomposition for smooth forms.

\begin{prop}\label{HodgeCinf} 
Let $1\le h\le 4$ and $\alpha\in \mathcal{D}(\G,E_0^h)$. There exist a smooth $(h-1)$-form $\phi$ and a smooth $(h+1)$-form $\zeta$ such that
\begin{equation}\label{eqqq}
\alpha=d_c\phi+\delta_c\zeta,
\end{equation}
and the forms $d_c\phi, \delta_c\zeta$ are in $L^p$, for every $1<p<\infty$.

Moreover, the decomposition is unique and, if $s_h,r_h$ are defined as in \eqref{indici}, we have $\phi\in \left(C^\infty\cap BL^{s_h,p}(\G,E_0^{h-1})\right)\cap\ker\,\delta_c$ and $\zeta\in \left(C^\infty\cap BL^{r_h,p}(\G,E_0^{h+1})\right)\cap\ker\,d_c$ and
\begin{equation}\label{estimateCinfy}
\|\phi\|_{\BLh{s_h}{p}{h-1}}+\|\zeta\|_{\BLh{r_h}{p}{h+1}}\leq c_{p,h}\|\alpha\|_{L^p(\G,E_0^h)}.
\end{equation}
\end{prop}
\begin{proof}
\noindent The uniqueness of the decomposition follows from Proposition \ref{Liouville}. Indeed,  suppose that $\phi,\zeta$ are smooth forms and that $d_c\phi, \delta_c\zeta \in L^p$ satisfy $d_c\phi+\delta_c\zeta=0$. Then $\delta_cd_c\phi=d_c\delta_c\zeta=0$, and trivially $d_c^2\phi=\delta_c^2\zeta=0$, which implies that
$$\Delta_{R,h}d_c\phi=\Delta_{R,h}\delta_c\zeta=0.$$
Hence, by Proposition \ref{Liouville},  the forms $d_c\phi, \delta_c\zeta$ are vector-valued polynomials. Since $d_c\phi, \delta_c\zeta\in L^{p}$, this implies $d_c\phi=\delta_c\zeta=0$, which proves the uniqueness.

In order  to show the existence of the decomposition,
let us first assume $\alpha\in \mathcal{D}(\G,E_0^2)$. By Theorem \ref{teo.kernelCartan23},  $\alpha=\Delta_{R,2}\Delta_{R,2}^{-1} \alpha$ and hence 
$$\dsy\alpha=\Delta_{R,2}\Delta_{R,2}^{-1} \alpha=d_c\left(\delta_c d_c\delta_c\Delta_{R,2}^{-1} \alpha\right)+\delta_c\left((d_c\delta_c)^2 d_c\Delta_{R,2}^{-1} \alpha\right).$$
The decomposition  follows once we set $\phi:=\delta_c d_c\delta_c\Delta_{R,2}^{-1} \alpha$ and $\zeta:=(d_c\delta_c)^2 d_c\Delta_{R,2}^{-1} \alpha$. 
Recall that $d_c:E_0^1\to E_0^2,\delta_c:E_0^3\to E_0^2$ are differential operators of order $3,2$, respectively. Thus, by Theorem \ref{teo.kernelCartan23}, the operator $\delta_c d_c\delta_c\Delta_{R,2}^{-1} $ is associated to a kernel of type $3$ and $ (d_c\delta_c)^2 d_c\Delta_{R,2}^{-1}$ to a kernel of type $2$. Now, denoting by $X^I,X^J$ horizontal derivatives of order, respectively, $3$ and $2$, it follows that $X^I\delta_c d_c\delta_c\Delta_{R,2}^{-1}, X^J(d_c\delta_c)^2 d_c\Delta_{R,2}^{-1}$ are kernels of type $0$. Hence, by Proposition \ref{folland zero}, there exists a constant $c=c(p,h)$ such that
$$
\|X^I\phi\|_{L^{p}(\G,E_0^{h-1})}\le c\|\alpha\|_{L^p(\G,E_0^h)}\quad\text{and}\quad \|X^J\zeta\|_{L^{p}(\G,E_0^{h+1})}\le c\|\alpha\|_{L^p(\G,E_0^h)},
$$
proving the result and the estimate \eqref{estimateCinfy}. 

The proof for  $h\neq 2$ can be proved analogously, possibly replacing Theorem \ref{teo.kernelCartan23} by Theorem \ref{teo.inverselaplHn} when $h=1,4$.
\end{proof}

Arguing as above, we can prove also the following ``homogeneous'' decomposition.
\begin{prop}\label{hodgeintrinsic}
Let $1\leq h\leq 4$ and $\alpha\in \mathcal{D}(\G,E_0^h)$. Then,
\begin{enumerate}
%
\item if $h=1$, there exist $\phi\in \ker\,\delta_c\cap C^\infty\cap BL^{3,p}(\G, E_0^{h-1})$ and  $\zeta\in \ker\,d_c\cap C^\infty\cap BL^{3,p}(\G, E_0^{h+1})$ such that 
$$\alpha=d_c\delta_cd_c\phi+\delta_c\zeta.$$
\item if $h=2$, there exist $h$-form $\phi\in \ker\,d_c$ and $(h+1)$-form $\zeta\in \ker\,d_c$, and both in $C^\infty\cap BL^{6,p}$, so that
$$\alpha=d_c\delta_c\phi+\delta_cd_c\delta_c\zeta.$$
\item if $h=3$, there exist $(h-1)$-form $\phi\in \ker\,\delta_c$ and $h$-form $\zeta\in \ker\,\delta_c$, both  in $C^\infty\cap BL^{6,p}$, so that
$$\alpha=d_c\delta_cd_c\phi+\delta_cd_c\zeta.$$
\item if $h=4$, there exist $(h-1)$-form $\phi\in \ker\,\delta_c$ and $(h+1)$-form $\zeta\in \ker\,d_c$, both  in $C^\infty\cap BL^{3,p}$, so that
$$\alpha=d_c\phi+\delta_cd_c\delta_c\zeta.$$
\end{enumerate}

Moreover, the decompositions are unique and, recalling that $M_h$ is the order of the Laplacian $\Delta_{M,h}$, we have the estimate
\begin{equation}\label{estimateCinfy2}
\|\phi\|_{BL^{M_h/2,p}(\G,E_0^\bullet)}+\|\zeta\|_{BL^{M_h/2,p}(\G,E_0^\bullet)}\leq c_{p,h}\|\alpha\|_{L^p(\G,E_0^\bullet)}.
\end{equation}
\end{prop}

\medskip
By density, we can obtain a more general result when $\alpha\in L^p$.

\begin{teo}\label{Lphodge} Let $1\le h\le 4, 1<p<\infty$ and let $s_h,r_h$ as in \eqref{indici}.
If $\alpha\in L^p(\G,E_0^h)$, there exist $\phi\in\BLh{s_h}{p}{h-1}\cap\ker\,\delta_c$ and $\zeta\in\BLh{r_h}{p}{h+1}\cap\ker\,d_c$ such that
$$\alpha=d_c\phi+\delta_c\zeta,$$
and the forms $d_c\phi, \delta_c\zeta$ are unique.
Hence, 
\begin{equation}\label{decomp0}
L^p(\G,E_0^h)=d_c\left(BL^{s_h,p}(\G, E_0^{h-1})\right)\,\oplus\, \delta_c\left(BL^{r_h,p}(\G, E_0^{h+1})\right).
\end{equation}

Moreover, we have the continuity estimate
\begin{equation}\label{estimateBL}
\|\phi\|_{\BLh{s_h}{p}{h-1}}+\|\zeta\|_{\BLh{r_h}{p}{h+1}}\leq c_{p,h}\|\alpha\|_{L^p(\G,E_0^h)}.
\end{equation} 
\end{teo}
\begin{proof} 
The uniqueness can be proved analogously to the Proposition \ref{HodgeCinf}. Indeed, in the sense of distribution we have
$$\Delta_{R,h}d_c\phi=\Delta_{R,h}\delta_c\zeta=0.$$
Hence, $d_c\phi, \delta_c\zeta$ are harmonic in the usual sense by hypoellipticity of the Laplacian. By Proposition \ref{Liouville}, we have that $d_c\phi, \delta_c\zeta$ are vector-valued polynomials. Since $d_c\phi, \delta_c\phi\in L^{p}$, we infer that $d_c\phi=\delta_c\zeta=0$, proving the uniqueness.

Let us prove the existence. Let $\alpha\in L^p(\G,E_0^h)$ and let $(\alpha_k)_{k\in\N}\subseteq \mathcal{D}(\G,E_0^h)$ be such that $\alpha_k\to\alpha$ in $L^p(\G,E_0^h)$. By Proposition \ref{HodgeCinf}, there exist $(\phi_k)_{k\in\N}\subseteq \left(C^\infty\cap BL^{s_h,p}\right)(\G, E_0^{h-1})$ and $(\zeta_k)_{k\in\N}\subseteq \left(C^\infty\cap BL^{r_h,p}\right)(\G, E_0^{h+1})$ such that
$\alpha_k=d_c\phi_k+\delta_c\zeta_k$ and \eqref{estimateCinfy} hold for every $k\in\N$. Then, for $k\gg 1$,
$$\|\phi_k\|_{\BLh{s_h}{p}{h-1}}+\|\zeta_k\|_{\BLh{r_h}{p}{h+1}}\leq c_{p,h}\|\alpha_k\|_{L^p}\leq c_{p,h}(\|\alpha\|_{L^p}+1),$$
showing that $(\phi_k)_{k\in\N}$ and $(\zeta_k)_{k\in\N}$ are bounded sequences in $\BLh{s_h}{p}{h-1}$ and $\BLh{r_h}{p}{h+1}$, respectively. By Banach-Alaoglu Theorem, up to  subsequences, there exist $\phi\in\BLh{s_h}{p}{h-1}, \zeta\in \BLh{r_h}{p}{h+1}$ such that $\phi_k\rightharpoonup\phi$ in $\BLh{s_h}{p}{h-1}$ and $\zeta_k\rightharpoonup\zeta$ in $\BLh{r_h}{p}{h+1}$. By Remark \ref{convweak},
$$d_c\phi_k\rightharpoonup d_c\phi\,\,\,\,\text{in $L^{p}(\G,E_0^{h})$}\quad\text{and}\quad \delta_c\zeta_k\rightharpoonup\delta_c\zeta\,\,\,\,\text{in $L^{p}(\G,E_0^{h})$},$$
and hence
$$\alpha=d_c\phi+\delta_c\zeta.$$

The estimate \eqref{estimateBL} follows by \eqref{estimateCinfy} and by the lower semicontinuity of the norm. Indeed, 
\begin{align*}
\|\phi\|_{\BLh{s_h}{p}{h-1}}&\leq\liminf_{k\to+\infty} \|\phi_k\|_{\BLh{s_h}{p}{h-1}}\\
&\leq c_{p,h} \liminf_{k\to+\infty}  \|\alpha_k\|_{L^p(\G,E_0^h)}= c_{p,h}\|\alpha\|_{L^p(\G,E_0^h)}.
\end{align*}
Analogously, we can argue for $\zeta$, completing the proof.
\end{proof}

Analogously, from  Proposition \ref{hodgeintrinsic} we obtain 
the following homogeneous decomposition. 

\begin{teo}\label{Lphodge2}
Let $1\leq h\leq 4$ and $1<p<\infty$. If $\alpha\in L^p(\G,E_0^h)$, then
\begin{enumerate}
\item if $h=1$, there exist $(h-1)$-form $\phi\in \ker\,\delta_c$ and $(h+1)$-form $\zeta\in \ker\,d_c$ such that both are in $BL^{3,p}$ and $\alpha=d_c\delta_cd_c\phi+\delta_c\zeta.$ Hence
\begin{equation*}
L^p(\G,E_0^h)=d_c\delta_cd_c\left(BL^{3,p}(\G, E_0^{h-1})\right)\,\oplus\, \delta_c\left(BL^{3,p}(\G, E_0^{h+1})\right).
\end{equation*}
\item if $h=2$, there exist $h$-form $\phi\in \ker\,d_c$ and $(h+1)$-form $\zeta\in \ker\,d_c$ such that both are in $BL^{6,p}$ and
$\alpha=d_c\delta_c\phi+\delta_cd_c\delta_c\zeta.$ Hence
\begin{equation*}
L^p(\G,E_0^h)=d_c\delta_c\left(BL^{6,p}(\G, E_0^{h})\right)\,\oplus\, \delta_cd_c\delta_c\left(BL^{6,p}(\G, E_0^{h+1})\right).
\end{equation*}
\item if $h=3$, there exist $(h-1)$-form $\phi\in \ker\,\delta_c$ and $h$-form $\zeta\in \ker\,\delta_c$ such that both are in $BL^{6,p}$ and
$\alpha=d_c\delta_cd_c\phi+\delta_cd_c\zeta.$ Hence
\begin{equation*}
L^p(\G,E_0^h)=d_c\delta_cd_c\left(BL^{6,p}(\G, E_0^{h-1})\right)\,\oplus\, \delta_cd_c\left(BL^{6,p}(\G, E_0^{h})\right).
\end{equation*}
\item if $h=4$, there exist $(h-1)$-form $\phi\in \ker\,\delta_c$ and $(h+1)$-form $\zeta\in \ker\,d_c$ such that both are in $BL^{3,p}$ and
$\alpha=d_c\phi+\delta_cd_c\delta_c\zeta.$ Hence
\begin{equation*}
L^p(\G,E_0^h)=d_c\left(BL^{3,p}(\G, E_0^{h-1})\right)\,\oplus\, \delta_cd_c\delta_c\left(BL^{3,p}(\G, E_0^{h+1})\right).
\end{equation*}
\end{enumerate}

Moreover, the decompositions are unique and, recalling that $M_h$ is the order of the Laplacian $\Delta_{M,h}$, we have 
\begin{equation*}
\|\phi\|_{BL^{M_h/2,p}(\G,E_0^\bullet)}+\|\zeta\|_{BL^{M_h/2,p}(\G,E_0^\bullet)}\leq c_{p,h}\|\alpha\|_{L^p(\G,E_0^\bullet)}.
\end{equation*}

\end{teo}

\begin{oss}\label{h05}
If $h=5$ and $\alpha\in L^p(\G,E_0^5)$, there exists $\phi\in BL^{1,p}(\G,E_0^4)\cap \ker\,\delta_c$ so that $\alpha=d_c\phi$ and $\|\phi\|_{BL^{1,p}}\leq c_p\|\alpha\|_{L^p}$.
\end{oss}

\medskip

\begin{oss}\label{coroLphodge}
Let $1\leq h\leq 4,\,1<p,q<\infty$ be such that $\frac{1}{p}+\frac{1}{q}=1$. 
Then,
\begin{enumerate}
\item with the same notations of Theorem \ref{Lphodge}, if $\alpha=d_c\alpha'+\delta_c\alpha''\in L^p(\G,E_0^h)$ and $\beta=d_c\beta'+\delta_c\beta''\in L^q(\G,E_0^h)$, then
\begin{equation}\label{adjoint}
\Scal{d_c\alpha'}{\delta_c\beta''}=0,\quad\Scal{d_c\alpha'}{\beta}=\Scal{\alpha}{d_c\beta'},\quad\Scal{\delta_c\alpha''}{\beta}=\Scal{\alpha}{\delta_c\beta''}.
\end{equation}
\item the spaces $d_c\left(BL^{s_h,p}(\G, E_0^{h-1})\right)$ and $\delta_c\left(BL^{r_h,p}(\G, E_0^{h+1})\right)$ are closed under $L^p$-convergence;
\end{enumerate}
\end{oss}
\begin{proof}
Proof of (1). 

Let $\alpha,\beta$ be as in the hypothesis. We consider $(\alpha_k)_{k\in\N}\subseteq \mathcal{D}(\G,E_0^h)$ such that $\alpha_k\to\alpha$ in $L^p$. According to the previous results, we can decompose $\alpha_k$ as $\alpha_k=d_c\alpha_k'+\delta_c\alpha_k''$, for every $k\in\N$. Thus,
$$\langle d_c\alpha_k'|\delta_c\beta''\rangle=\langle d_c^2\alpha_k'|\beta''\rangle=0\quad\text{for every $k\in\N$}.$$
Hence,
\begin{align*}
|\langle d_c\alpha'|\delta_c\beta''\rangle|&=|\langle d_c\alpha'-d_c\alpha'_k|\delta_c\beta''\rangle|\\
&\leq \|d_c\alpha'-d_c\alpha'_k\|_{L^p(\G,E_0^h)}\,\|\delta_c\beta''\|_{L^q(\G,E_0^h)}\\ 
&\stackrel{\eqref{estimateBL}}{\leq}\|\alpha-\alpha_k\|_{L^p(\G,E_0^h)}\,\|\beta\|_{L^q(\G,E_0^h)}\xrightarrow{k\to+\infty} 0,
\end{align*}
showing that $\Scal{d_c\alpha'}{\delta_c\beta''}=0$.

Let us prove $\Scal{d_c\alpha'}{\beta}=\Scal{\alpha}{d_c\beta'}$. (the proof of $\Scal{\delta_c\alpha''}{\beta}=\Scal{\alpha}{\delta_c\beta''}$ is analogous). By previous computations, we already know that
$\Scal{d_c\alpha'}{\delta_c\beta''}=\Scal{\delta_c\alpha''}{d_c\beta'}=0.$
Thus, we have
\begin{align*}
\Scal{d_c\alpha'}{\beta}&=\Scal{d_c\alpha'}{d_c\beta'+\delta_c\beta''}=\Scal{d_c\alpha'}{d_c\beta'}\\
&=\Scal{d_c\alpha'+\delta_c\alpha''}{d_c\beta'}=\Scal{\alpha}{d_c\beta'}.
\end{align*}

\noindent Proof of (2).

Let $(d_c\phi_k)_{k\in\N}\subseteq d_c\left(BL^{s_h,p}(\G, E_0^{h-1})\right)$ be such that $d_c\phi_k\to\omega$ in $L^p(\G,E_0^h)$. There exists $\phi\in BL^{s_h,p}(\G,E_0^{h-1})$ such that $\omega=d_c\phi$. Indeed, 
by Theorem \ref{Lphodge}, we write $\omega=d_c\phi+\delta_c\zeta$, where $\phi\in BL^{s_h,p}(\G,E_0^{h-1})$.
Let us show that $d_c\phi_k\to d_c\phi$ in $\mathcal{D}'(\G,E_0^h)$.
For every $\varphi\in \mathcal{D}(\G,E_0^h)$, by Proposition \ref{HodgeCinf}, we can write $\varphi=d_c\varphi'+\delta_c\varphi''$ and  \eqref{estimateCinfy} holds for every exponent $1<q<\infty$. By (1), $\langle \delta_c\zeta|d_c\varphi'\rangle=0$, then
\begin{align*}
|\langle d_c\phi_k-d_c\phi|\varphi\rangle|&=|\langle d_c\phi_k-d_c\phi|d_c\varphi'+\delta_c\varphi''\rangle| 
= |\langle d_c\phi_k-d_c\phi|d_c\varphi'\rangle|\\
&= |\langle d_c\phi_k-d_c\phi-\delta_c\zeta|d_c\varphi'\rangle|\\
&\leq \|d_c\phi_k-\omega\|_{L^p(\G,E_0^h)}\,\|d_c\varphi'\|_{L^q(\G,E_0^h)}\\
&\stackrel{\eqref{estimateCinfy}}{\leq} \|d_c\phi_k-\omega\|_{L^p(\G,E_0^h)}\,\|\varphi\|_{L^q(\G,E_0^h)}\xrightarrow{k\to+\infty} 0,
\end{align*}
where $\frac{1}{p}+\frac{1}{q}=1$. Since $d_c\phi_k\to\omega$ in $L^p$, we have proved that $\omega=d_c\phi$.
\end{proof}

\bigskip

We now turn to the case $p=1$. It is shown in \cite{BaldoOrlandi98} that already in $\R^n$ an $L^1$-Hodge decomposition analogous to \eqref{decomp0} cannot hold (see Example 2.6 therein). In our setting as well, only a weak form of the Hodge decomposition in $\mathcal{D}'$ can be obtained for forms in $L^1$.

\begin{prop}\label{L1hodge} Let $1\le h\le 4$ and $s_h, r_h$ as in \eqref{indici}.
For every $\alpha\in L^1(\G,E_0^h)$, there exist $\phi\in L^{Q/(Q-s_h),\infty}(\G,E_0^{h-1})\cap\ker\,\delta_c$ and $\zeta\in L^{Q/(Q-r_h),\infty}(\G,E_0^{h+1})\cap\ker\,d_c$ such that
$$\alpha=d_c\phi+\delta_c\zeta\quad\text{in $\mathcal{D}'(\G,E_0^h)$},$$
and the forms $d_c\phi, \delta_c\zeta$ are unique.

Moreover, the continuity estimates hold:
\begin{equation}\label{estimateL1Hodge}
\|\phi\|_{L^{Q/(Q-s_h),\infty}(\G,E_0^{h-1})}+\|\zeta\|_{L^{Q/(Q-r_h),\infty}(\G,E_0^{h+1})}\leq c\|\alpha\|_{L^1(\G,E_0^h)}.
\end{equation} 
\end{prop}
\begin{proof} 
The uniqueness follows from an argument again related to Proposition \ref{HodgeCinf}. To prove the existence, we limit ourselves to the case $h=2$. Let $\alpha\in L^1(\G,E_0^2)$ and suppose for a moment that the decomposition is proved. Then, by Theorem \ref{teo.kernelCartan23}, $K_1:=\delta_cd_c\delta_c\Delta_{R,2}^{-1}$ $K_2:=(d_c\delta_c)^2 d_c \Delta_{R,2}^{-1}
$ are kernels of type $3,2$, respectively. Thus, by Lemma \ref{lemma.stimaconvL1supp},
$$\phi:=\alpha\ast K_1\in L^{Q/(Q-s_h),\infty}(\G,E_0^{h-1}), \quad\zeta:=\alpha\ast K_2\in L^{Q/(Q-r_h),\infty}(\G,E_0^{h+1}),$$ 
and the estimate \eqref{estimateL1Hodge} holds.

Let us now prove the decomposition. Let $(\alpha_k)_{k\in\N}\subseteq \mathcal{D}(\G,E_0^h)$ be such that $\alpha_k\to\alpha$ in $L^1$. Then, by Proposition \ref{HodgeCinf} and Lemma \ref{lemma.stimaconvL1supp}, we have
\begin{equation}\label{decompo}
\alpha_k=d_c\phi_k+\delta_c\zeta_k,\quad\text{for every $k\in\N$}
\end{equation}
where $\phi_k:=\alpha_k\ast K_1, \zeta_k:=\alpha_k\ast K_2\in L_{loc}^1$. For every $\varphi\in \mathcal{D}(\G,E_0^{h-1})$, using Lemma \ref{lemma.estimatekernelpol}, we get $\varphi\ast \checkV{K_1}\in L^\infty$. Then,
\begin{align*}
|\langle\phi_k-\phi|\varphi\rangle|&=|\langle(\alpha_k-\alpha)\ast K_1|\varphi\rangle|\stackrel{\text{Lemma}\, \eqref{lemma.checkconvkernel}}{\leq} |\langle \alpha_k-\alpha|\varphi\ast \checkV{K_1}\rangle| \\
&\leq \|\alpha_k-\alpha\|_{L^1(\G,E_0^h)}\,\|\varphi\ast\checkV{K_1}\|_{L^\infty(\G,E_0^h)}\xrightarrow{k\to+\infty} 0,
\end{align*}
showing that $\phi_k\to\phi$ in $\mathcal{D}'(\G,E_0^{h-1})$. Analogously, $\zeta_k\to\zeta$ in $\mathcal{D}'(\G,E_0^{h+1})$. Thus, the decomposition holds passing to the limit in \eqref{decompo}.  Moreover, since $\phi_k$ is $\delta_c$-closed and $\zeta_k$ is $d_c$-closed for every $k\in\N$, then also $\phi,\zeta$ are $\delta_c,d_c$-closed, respectively.
\end{proof}

\bigskip

\subsection{Global $L^p$-Poincaré inequalities}\label{sec.PoincareP}

\subsubsection{The case $p>1$} The following global $L^p$-Poincaré inequalities hold.

\begin{teo}\label{teo.globalLpPoincineq}
Let $h=1,\ldots,4$ and let $s_h$ as in \eqref{indici}. Suppose $1<p<Q/s_h$ and let $q_h$ such that $\dsy\frac{1}{q_h}=\frac{1}{p}-\frac{s_h}{Q}$. 
Then, there exists a constant $C=C(Q,p,h)>0$ such that for every $d_c$-closed $h$-form $\alpha\in L^p(\G,E_0^h)$, there exists $\phi\in BL^{s_h,p}(\G,E_0^{h-1})$ satisfying $d_c\phi=\alpha$ and 
$$\|\phi\|_{BL^{s_h,p}(\G,E_0^{h-1})}\leq C\|\alpha\|_{L^p(\G,E_0^h)}.$$

Moreover, 
\begin{equation}\label{poinC}
\|\phi\|_{L^{q_h}(\G,E_0^{h-1})}\leq C\|\alpha\|_{L^p(\G,E_0^h)}.
\end{equation}
\end{teo}
\begin{proof}
Let $\alpha\in L^p(\G,E_0^h)$ be a $d_c$-closed form. By  Theorem \ref{Lphodge} there exist $\phi\in BL^{s_h,p}(\G,E_0^{h-1})$ and $\zeta\in BL^{r_h,p}(\G,E_0^{h+1})$ so that
$$\alpha=d_c\phi+\delta_c\zeta.$$
Since $\alpha$ is $d_c$-closed then $d_c\delta_c\zeta=0$ and, trivially, $\delta_c(\delta_c\zeta)=0$. Thus, $\Delta_{R,h}\delta_c\zeta=0$. Since $\delta_c\zeta\in L^{p}(\G,E_0^h)$, Proposition \ref{Liouville} yields $\delta_c\zeta=0$.
Hence, $\alpha=d_c\phi$. Using  the estimate \eqref{estimateBL} together with the embedding result in Proposition \ref{car_spazio}, we conclude the proof.
\end{proof}

\begin{oss}
Using Remark \ref{h05}, we have also proved the $L^p$-Poincaré inequality for $h=5$. Indeed, if $\alpha\in L^p(\G,E_0^5)$, then $\alpha$ is automatically $d_c$-closed and we obtain the inequality \eqref{poinC} with $\phi\in BL^{1,p}(\G,E_0^4)\hookrightarrow L^q(\G,E_0^4)$ with $\frac{1}{q}=\frac{1}{p}-\frac{1}{Q}$.
\end{oss}

\subsubsection{The case $p=1$}
We deal now with the case $p=1$. Starting from a $d_c$-closed form $\alpha\in L^1$ we want to show that 
the primitives $\phi$  defined  as
\begin{equation}\label{primitive}\phi:=\left\{
\begin{array}{ll}
\delta_c(d_c\delta_c)^2\Delta_{R,h}^{-1}\alpha,\quad&\text{if $h=1,3$},\\
\delta_cd_c\delta_c\Delta_{R,2}^{-1}\alpha,\quad&\text{if $h=2$},\\
\delta_c\Delta_{R,4}^{-1}\alpha,\quad&\text{if $h=4$}.
\end{array}\right.
\end{equation}
satisfy the  estimates stated in Theorem \ref{P1} below.

Unfortunately, we cannot proceed as in the case $p>1$, since only a weak version of the $L^1$-decomposition is available.
The $L^1$-Poincaré estimate is obtained via a different approach, in the spirit of \cite{BFP3}. It relies on a recent result obtained in \cite{BT}  that in our case reads as follows.

\begin{teo}[see Theorem $5.4$ in \cite{BT}]\label{teo.GNineqCartan} Let $h=0,\ldots, 4$.
There exists $C>0$ such that for every $\phi\in \mathcal{D}(\G,E_0^h)$ with $\delta_c\phi=0$, we have
\begin{align*}
\|\phi\|_{L^{Q/(Q-1)}(\G,E_0^0)}&\leq C\|d_c \phi\|_{L^1(\G,E_0^1)},\quad\quad&\text{if $h=0$}; \\
\|\phi\|_{L^{Q/(Q-3)}(\G,E_0^1)}&\leq C\|d_c \phi\|_{L^1(\G,E_0^2)},\quad\quad&\text{if $h=1$}; \\
\|\phi\|_{L^{Q/(Q-2)}(\G,E_0^2)}&\leq C\|d_c \phi\|_{L^1(\G,E_0^3)},\quad\quad&\text{if $h=2$}; \\
\|\phi\|_{L^{Q/(Q-3)}(\G,E_0^3)}&\leq C\|d_c \phi\|_{L^1(\G,E_0^4)},\quad\quad&\text{if $h=3$}; \\
\|\phi\|_{L^{Q/(Q-1)}(\G,E_0^4)}&\leq C\|d_c \phi\|_{L^1(\G,E_0^5)},\quad\quad&\text{if $h=4$}.
\end{align*}
\end{teo}

 Recall that, by Theorem \ref{teo.kernelCartan23}, the convolution operator 
\begin{equation}\label{k-h}
 K_h:=\left\{
\begin{array}{ll}
\delta_c(d_c\delta_c)^2\Delta_{R,h}^{-1},\quad&\text{if $h=1,3$}, \\
\delta_cd_c\delta_c\Delta_{R,h}^{-1},\quad&\text{if $h=2$},\\
\delta_c\Delta_{R,h}^{-1},\quad&\text{if $h=4$},\\
\end{array}\right.
\end{equation}
associated with the kernel of type 
$$\mu_h=\left\{
\begin{array}{ll}1\quad&\text{if $h=1$}, \\
3\quad&\text{if $h=2,4$}, \\
2\quad&\text{if $h=3$}.
\end{array}\right.
$$

\begin{lemma}\label{lemma.ugzerodistrseseee}
Let $\alpha\in L^1(\G,E_0^h)\cap \ker\,d_c$ and let $ K_h$ be the convolution operator kdefined in \eqref{k-h}.
Then
 $$d_c K_h\alpha=\alpha\,,\quad\quad \delta_cK_h\alpha=0,$$ in distributional sense.
\end{lemma}
\begin{proof}
 For any $h$ the  kernels associated to $ K_h$ are in $ L^1_{loc}$  and will be denoted again by $K_h$. 
Consider
 a sequence $(\alpha_k)_{k\in\mathbb N}$
of compactly supported smooth forms that converges to $\alpha $ in $L^1(\G, E_0^h)$. Thanks to Lemma \ref{lemma.stimaconvL1supp}, as $k\to\infty$ it holds
$$
K_h\alpha_k \to   K_h \alpha\qquad\text{in $ L^1_{\mathrm{loc}}(\G,E_0^{h-1})$.}
$$
Hence $  K_h \alpha_k \to   K_h \alpha$ in $ \mc D'(\G,E_0^{h-1})$ together with all their derivatives. In particular
\begin{equation}\label{dc* tilde K} \delta_c  K_h\alpha = \lim_{k\to\infty} \delta_c   K_h\alpha_k=0,
\end{equation}
by the very definition of $K_h$ keeping in mind that $\delta_c\circ \delta_c=0$.

We now prove that, if we set
\begin{equation*}
\widetilde P_h:=\left\{
\begin{array}{ll}
\delta_cd_c,\quad&\text{if $h=1$}, \\
(\delta_cd_c)^3,\quad&\text{if $h=2,4$},\\
(\delta_cd_c)^2,\quad&\text{if $h=3$},\\
\end{array}\right.
\end{equation*}
then  \begin{equation}\label{parte zero}
\lim_{k\to +\infty}\widetilde{P_h}\Delta_{R,h}^{-1}\alpha_k=0\quad\text{in $\mathcal{D}'(\G,E_0^{h})$}
\end{equation}
Notice that 
for compactly supported forms this is exactly the part of the component of $\Delta_{R,h}$ annihilating compactly supported $d_c$-closed forms, but here we do not know that $\alpha_k$ are $d_c$-closed.

To avoid cumbersome notations, let us prove \eqref{parte zero} just for $h=2$.
By Lemma \ref{lemma.commDeltadcdeltac}- iii) and by Lemma \ref{lemma.checkconvkernel}, for all $\phi\in \mc D(\G,E_0^2)$, 
\begin{gather}\label{dai2}
\begin{split}
\lim_{k\to\infty} \Scal{\widetilde P_2 \Delta_{R,2}^{-1}\alpha_k}{\phi}&=\lim_{k\to\infty} \Scal{(\delta_cd_c)^3 \Delta_{R,2}^{-1}\alpha_k}{\phi}  = \lim_{k\to\infty} \Scal{ (\delta_cd_c)^2\delta_c\Delta_{R,3}^{-1}d_c\alpha_k}{ \phi}\\
&= \lim_{k\to\infty} \Scal{ d_c\alpha_k}{ \phi\ast \checkV{K_3}} = \lim_{k\to\infty}  \Scal{ \alpha_k}{ \delta_c( \phi\ast \checkV{K_3})}.
\end{split}
\end{gather}
On the other hand, since $K_3$ is a kernel of type 2, $\phi\in\mathcal{D}$ and $\delta_c$ is a homogeneous differential operator of order $2$, then, by equation \ref{eq.convspostcose}, $\delta_c(\phi\ast\ccheck K_3)=\ccheck\delta_c \ccheck\phi\ast\ccheck K_3$. Moreover, by Lemma \ref{lemma.estimatekernelpol}, we have $\ccheck\delta_c \ccheck\phi\ast\ccheck K_3\in L^\infty$.
Hence, we have proved
$$\lim_{k\to\infty} \Scal{\widetilde P_2 \Delta_{R,2}^{-1}\alpha_k}{\phi}=\lim_{k\to\infty}  \Scal{ \alpha_k}{ \delta_c( \phi\ast \checkV{K_3})}=\Scal{ \alpha}{  \delta_c( \phi\ast \ccheck  K_3)}.$$
Let us show that
\begin{equation}\label{closed eq:1}
\int_\G \langle \alpha,  \delta_c( \phi\ast \ccheck  K_3)\rangle\, dx=0.
\end{equation}
If $k\in \mathbb N$, let $(\sigma_k)_{k\in\N}\subseteq \mathcal{D}(\G)$ be a sequence cut-off functions supported
in $B(e,2k)$ and identically 1 on $B(e,k)$ such that $k |\nabla_\G\sigma_k|+k^2|\nabla_\G^2\sigma_k|+k^3|\nabla_\G^3\sigma_k|\leq c$, for every $k$.
By Lemma \ref{leibniz} and the dominated convergence theorem, keeping in mind  $d_c\alpha=0$, 
\begin{align*}
\int_\G \langle \alpha,  \delta_c( \phi\ast \ccheck  & K_3)\rangle\, dx  = \lim_{k\to\infty}\int_\G \scal{\alpha}{ \sigma_k \delta_c( \phi\ast \ccheck  K_3)}\, dx\\
&=\lim_{k\to\infty} \Big(\int_\G \scal{\alpha}{\delta_c(\sigma_k( \phi\ast \ccheck  K_3))}\, dx
- \int_\G \scal{\alpha}{ [\delta_c, \sigma_k](\phi\ast \ccheck  K_3)}\, dx\Big)\\
&
=\lim_{k\to\infty} \Big(\Scal{d_c\alpha}{ \sigma_k(\phi\ast \ccheck  K_3)}
-  \int_\G \scal{\alpha}{ [\delta_c, \sigma_k](\phi\ast \ccheck  K_3)}\, dx\Big)
\\&
=
-\lim_{k\to\infty} \Big(\int_\G \scal{\alpha}{ [\delta_c, \sigma_k](\phi\ast \ccheck  K_3)}\, dx\Big)=0
\end{align*}
since the horizontal derivatives of any order of $\sigma_k$ vanish as
$k\to\infty$ and $\phi \ast \ccheck K_3\in L^\infty$ by Lemma \ref{lemma.estimatekernelpol}. Hence \eqref{closed eq:1} is proved and the proof is completed.
The proof for $h\ne 2$ works analogously and is omitted.
\end{proof}

We are now able to prove the main estimates of this section.
We  begin with  a proposition contained in \cite{pansu-tripaldi} related to vanishing averages of a differential form. Recall that assuming that averages vanish allows a finer control of the effect of kernels. If $P$ is the operator of convolution with a kernel of type $\mu>0$, and $\omega\in L^1$, then the $L^1$ norm of $P\omega$ on shells $B(e,2R)\setminus B(e,R)$ is $O(R^\mu)$. If furthermore $\omega$ has vanishing average, this can be improved to $o(R^\mu)$, as follows from Theorem \ref{teo.avesterrorG}. As proved by Pansu and Tripaldi,  if a differential form $\omega\in L^1$ is closed, then all its averages vanish in the following sense. 
\begin{proposition}[see Proposition 3.3 in \cite{pansu-tripaldi}]\label{media nulla} Let $\omega\in L^1(\G, E_0^h)$ be a $d_c$-closed form and $\beta$ be
a left-invariant form of
complementary degree $N-h$
then
$$
\int_\G \omega\wedge\beta=0\,.
$$
\end{proposition}

Gathering all the above results and following the path of the proof of Theorem $5.2$  in \cite{BFP2}, we obtain the following result.

\begin{teo}\label{teo.mainestimateGcartan}
For every $d_c$-closed form $\alpha\in L^1(\G,E_0^h)$, with $h=1,\ldots,4$ we have
\begin{align*}
\|\delta_c(d_c\delta_c)^2\Delta_{R,1}^{-1}\alpha\|_{L^{Q/(Q-1)}(\G,E_0^0)}&\leq C\|\alpha\|_{L^1(\G,E_0^1)},\quad\quad&\text{if $h=1$}; \\
\|\delta_cd_c\delta_c\Delta_{R,2}^{-1}\alpha\|_{L^{Q/(Q-3)}(\G,E_0^1)}&\leq C\|\alpha\|_{L^1(\G,E_0^2)},\quad\quad&\text{if $h=2$}; \\
\|\delta_c(d_c\delta_c)^2\Delta_{R,3}^{-1}\alpha\|_{L^{Q/(Q-2)}(\G,E_0^2)}&\leq C\|\alpha\|_{L^1(\G,E_0^3)},\quad\quad&\text{if $h=3$}; \\
\|\delta_c\Delta_{R,4}^{-1}\alpha\|_{L^{Q/(Q-3)}(\G,E_0^3)}&\leq C\|\alpha\|_{L^1(\G,E_0^4)},\quad\quad&\text{if $h=4$}.
\end{align*}
\end{teo}
\begin{proof} Let $\alpha\in L^1(\G,E_0^h)\cap\ker\,d_c$. By Proposition \ref{media nulla}, $\alpha$ has zero average.
The convolution operator $ \delta_c(d_c\delta_c)^2\Delta_{R,1}^{-1}$ is associated to a kernel of type $1$ while $\delta_cd_c\delta_c\Delta_{R,2}^{-1}$ and $\delta_c\Delta_{R,4}^{-1}$ are associated to kernels of type $3$; finally, $\delta_c(d_c\delta_c)^2\Delta_{R,3}^{-1}$ is associated to a kernel of type $2$.

 Let $k\in\N$ and let $\chi_k$ be the cut-off function supported in $B(e,2k)$ such that $\chi_k\equiv 1$ in $B(e,k)$. For every $0<\e<1$, denote as $J_\e$ the usual Friedrichs' mollifier (with respect to the group structure). Hence,  define in $E_0^{h-1}$
\begin{align*}
v_{\e,k}&:=J_\e\ast\delta_c(\chi_k(d_c\delta_c)^2\Delta_{R,1}^{-1}\alpha),&\text{if $h=1$}; \\
v_{\e,k}&:=J_\e\ast\delta_c(\chi_k d_c\delta_c\Delta_{R,2}^{-1}\alpha),&\text{if $h=2$}; \\
v_{\e,k}&:=J_\e\ast\delta_c(\chi_k(d_c\delta_c)^2\Delta_{R,3}^{-1}\alpha),&\text{if $h=3$}; \\
v_{\e,k}&:=J_\e\ast\delta_c(\chi_k\Delta_{R,4}^{-1}\alpha),&\text{if $h=4$}.
\end{align*}
Let us denote by
\begin{equation*}
P_h:=\left\{
\begin{array}{ll}
(d_c\delta_c)^2,\quad&\text{if $h=1,3$}, \\
d_c\delta_c,\quad&\text{if $h=2$},\\
1,\quad&\text{if $h=4$}.\\
\end{array}\right.
\end{equation*}
Thus, we can write for any degree $h$,
$$v_{\e,k}=J_\e\ast\delta_c(\chi_k P_h\Delta_{R,h}^{-1}\alpha).$$
We observe at first that $\delta_c v_{\e,k}=J_\e\ast \delta_c^2(\chi_k P_h\Delta_{R,h}^{-1}\alpha)=0$, implying that $v_{\e,k}$ is $\delta_c$-closed. Moreover, $\delta_c(\chi_k P_h\Delta_{R,h}^{-1}\alpha)$ is compactly supported and uniformly bounded in $L^1(\G,E_0^{h-1})$. Indeed, by Corollary \ref{leibniz}
$$\delta_c(\chi_k P_h\Delta_{R,h}^{-1}\alpha)=\chi_k(\delta_cP_h\Delta_{R,h}^{-1}\alpha)+[\delta_c,\chi_k] (P_h\Delta_{R,h}^{-1}\alpha),$$
where, by Lemma \ref{lemma.stimaconvL1supp}, both the terms in the right hand side have corresponding kernels of type $\geq 1$ and $\chi_k$ has compact support. In order to treat simultaneously all the cases, we set also $s_h$ as in \eqref{indici}.
Applying Theorem \ref{teo.GNineqCartan} to $v_{\e,k}\in \mathcal{D}(\G,E_0^{h-1})$, and recalling that by Lemma \ref{lemma.ugzerodistrseseee},
$$d_c\delta_c P_h\Delta_{R,h}^{-1}\alpha=d_c K_h\alpha=\alpha,$$ we have
\begin{gather}\label{eq.estveN}
\begin{split}
\|&v_{\e,k}\|_{L^{Q/(Q-s_h)}(\G,E_0^{h-1})}\leq C \|d_c v_{\e,k}\|_{L^1(\G,E_0^{h})}\\
&\quad= C\|J_\e\ast d_c\delta_c(\chi_k P_h\Delta_{R,h}^{-1}\alpha)\|_{L^1(\G,E_0^{h})}\\
&\quad\leq C\left(\|J_\e\ast[d_c\delta_c,\chi_k] (P_h\Delta_{R,h}^{-1}\alpha)\|_{L^1(\G,E_0^{h})}+\|J_\e\ast\chi_k(d_c\delta_c P_h\Delta_{R,h}^{-1}\alpha)\|_{L^1(\G,E_0^{h})}\right)\\
&\quad\leq C\left(\|[d_c\delta_c,\chi_k] (P_h\Delta_{R,h}^{-1}\alpha)\|_{L^1(\G,E_0^{h})}+\|\chi_k(d_c\delta_c P_h\Delta_{R,h}^{-1}\alpha)\|_{L^1(\G,E_0^{h})}\right)\\
&\quad\leq C\left(\|[d_c\delta_c,\chi_k] (P_h\Delta_{R,h}^{-1}\alpha)\|_{L^1(\G,E_0^{h})}+\|\chi_k\alpha\|_{L^1(\G,E_0^{h})}\right),
\end{split}
\end{gather}

We are left to prove that the term $\|[d_c\delta_c,\chi_k] (P_h\Delta_{R,h}^{-1}\alpha)\|_{L^1(\G,E_0^{h})}=o(1)$ as $k\to +\infty$.

By Corollary \ref{leibniz}, we have that  $[d_c \delta_c, \chi_k]$ can be written as a sum of terms of the form $P_j^h(X^s)$ with $j=0,\cdots,5$, depending on $h$. 
By Proposition \ref{prop.kernelopcheck},  the norm  $\| [d_c \delta_c, \chi_k] ( P_h\Delta_{R,h}^{-1} \alpha)\|_{L^1(\G,E_0^h)}$ can be estimated
by a sum of terms of the form
$$
\frac{1}{k^\mu} \int_{B(e,2k)\setminus B(e,k)} \big|K\alpha \big| \, dx,
$$
where $K$ is a kernel of type $\mu\ge 1$. Thus, we can apply Theorem \ref{teo.avesterrorG} to conclude that
$$
\| [d_c \delta_c, \chi_k] ( P_h\Delta_{R,h}^{-1} \alpha)\|_{L^1(\G,E_0^h)}\longrightarrow 0\qquad \text{as $k\to \infty$.}
$$
If $\eps\to 0$, then $v_{\eps,k} \to \delta_c  (\chi_k P_h \Delta_{R,h}^{-1} \alpha)$ in $L^1(\G, E_0^{h-1})$ and hence, up to a subsequence, we may assume a.e. . Fatou theorem applied to \eqref{eq.estveN} provides an $L^{Q/(Q-s_h)}$ bound on $\delta_c(\chi_k P_h \Delta_{R,h}^{-1} \alpha)$. The term
$$
\delta_c(\chi_k P_h\Delta_{R,h}^{-1} \alpha) =  \chi_k\delta_c P_h \Delta_{R,h}^{-1} \alpha + [\delta_c, \chi_k] P_h\Delta_{R,h}^{-1} \alpha,
$$
converges a.e.\,to $\delta_c P_h \Delta_{R,h}^{-1} \alpha$. Again by Fatou theorem, 
$$
\|\delta_c P_h \Delta_{R,h}^{-1} \alpha\|_{L^{Q/(Q-s_h)}(\G,E_0^{h-1})}\le C \|\alpha\|_{L^{1}(\G,E_0^h)}.
$$
This completes the proof.
\end{proof}

\medskip
\begin{teo}\label{P1} 
Let $h=1,\ldots,4$ and set $q_h:=Q/(Q-s_h)$, where $s_h$ is defined in \eqref{indici}.
Then, there exists a constant $C=C(Q,h)>0$ such that, for every $d_c$-closed $h$-form $\alpha\in L^1(\G,E_0^h)$, there exists an $(h-1)$-form $\phi\in L^{q_h}(\G,E_0^{h-1})$ such that
$$d_c\phi=\alpha\quad\quad\text{and}\quad\quad\|\phi\|_{L^{q_h}(\G,E_0^{h-1})}\leq C\|\alpha\|_{L^1(\G,E_0^h)}.$$
\end{teo} 
\begin{proof}
Combining together Theorem \ref{teo.mainestimateGcartan} and Lemma \ref{lemma.ugzerodistrseseee}, we obtain the result, where the primitives $\phi$ are defined as in \eqref{primitive}, depending on $h$.
\end{proof}

\begin{remark} In \cite{pansu-tripaldi}, Remark 2.5, it is proved that any Carnot group is $Q$-parabolic, i.e.  for every
compact set $V$, there exists a smooth compactly supported function  equals $1$ on
$V$, whose gradient has arbitrarily small $L^Q$ norm. It follows, as in the Heisenberg setting (see Section 1.2 in \cite{BFP3}),  that an  $L^1$-Poincar\'e inequality cannot hold for   $\omega\in  E_0^5$ in the Cartan group .
\end{remark}

\bigskip

\subsection{Sobolev-Gaffney type inequalities}\label{sec:sobgaffineqGCartan}

 It is well known that for a scalar function,  if $\alpha\in \mc D(\G)$ $p>1$ and $s\ge 0$ then there exists a constant such that
$$\|\alpha\|_{W^{s+1,p}(\G)}\leq C \Big(\|d_c\alpha\|_{W^{s,p}(\G,E_0^{1})}+\|\alpha\|_{L^p(\G)}\Big),$$ (furthermore, the converse estimate also holds).
By Hodge duality, an analogous result follows for a smooth compactly supported $5$-form  $\alpha$, as
$$\|\alpha\|_{W^{s+1,p}(\G,E_0^5)}\leq C\Big(\|\delta_c\alpha\|_{W^{s,p}(\G,E_0^{4})}+\|\alpha\|_{L^p(\G,E_0^5)}\Big).$$

We conclude this section by a remark akin to Remark $5.3$ in \cite{BF7}, providing   Sobolev-Gaffney type inequalities for Sobolev spaces $W^{s,p}(\G,E_0^h)$, when $p>1$ and $h\ge 1$. These results give a more intrinsic perspective of the Sobolev spaces, using the ``adapted'' operators $d_c$ and $\delta_c$.

As discussed in Example $2.4$ of \cite{BaldoOrlandi98}, the Sobolev-Gaffney inequality  for $p=1$ already fails in $\R^n$. Indeed, one can construct a sequence $(\alpha_k)_{k\in\N}$ of smooth compactly supported $1$-forms in $\R^2$ such that 
$$\|d\alpha_k\|_{L^1(\R^2,\bigwedge^{2}T^*\R^2)}+\|\delta\alpha_k\|_{L^1(\R^2,\bigwedge^{0}T^*\R^2)}+\|\alpha_k\|_{L^1(\R^2,\bigwedge^{1}T^*\R^2)}\leq C$$
but
$$\lim_{k\to+\infty}\|\nabla\alpha_k\|_{L^1(\R^2,\bigwedge^{1}T^*\R^2)}=+\infty.$$
As a consequence, in general, a form $\alpha\in L^1(\R^n,\bigwedge^{h}T^*\R^n)$ may have $\delta\alpha\in L^1(\R^n,\bigwedge^{h-1}T^*\R^n)$ and $d\alpha\in L^1(\R^n,\bigwedge^{h+1}T^*\R^n)$, without belonging to the space $W^{1,1}(\R^n,\bigwedge^{h}T^*\R^n)$.

For $p>1$ the following estimates hold for a $h-$ form.

\begin{teo}\label{teo.SobGafftypeCarnot}
Let $s\in\N$ and $\,1<p<\infty$. For every $\alpha\in \mathcal{D}(\G,E_0^h)$ we have:
\begin{itemize}
\item if $h=1$,
\begin{align*}
\|\alpha\|_{W^{s+3,p}(\G,E_0^h)}\leq C\Big(&\|d_c\alpha\|_{W^{s,p}(\G,E_0^{h+1})}+\|\delta_cd_c\delta_c\alpha\|_{W^{s,p}(\G,E_0^{h-1})}\\
&\quad+\|\alpha\|_{L^p(\G,E_0^h)}\Big),
\end{align*}
\item if $h=2$,
\begin{align*}
\|\alpha\|_{W^{s+6,p}(\G,E_0^h)}\leq C\Big(&\|d_c\delta_cd_c\alpha\|_{W^{s,p}(\G,E_0^{h+1})}+\|d_c\delta_c\alpha\|_{W^{s,p}(\G,E_0^{h})}\\
&\quad+\|\alpha\|_{L^p(\G,E_0^h)}\Big),
\end{align*}
\item if $h=3$,
\begin{align*}
\|\alpha\|_{W^{s+6,p}(\G,E_0^h)}\leq C\Big(&\|\delta_cd_c\alpha\|_{W^{s,p}(\G,E_0^{h})}+\|\delta_cd_c\delta_c\alpha\|_{W^{s,p}(\G,E_0^{h-1})}\\
&\quad+\|\alpha\|_{L^p(\G,E_0^h)}\Big),
\end{align*}
\item if $h=4$,
\begin{align*}
\|\alpha\|_{W^{s+3,p}(\G,E_0^h)}\leq C\Big(&\|d_c\delta_cd_c\alpha\|_{W^{s,p}(\G,E_0^{h+1})}+\|\delta_c\alpha\|_{W^{s,p}(\G,E_0^{h-1})}\\
&\quad+\|\alpha\|_{L^p(\G,E_0^h)}\Big).
\end{align*}
\end{itemize}
\end{teo}
\begin{proof} The proof is very similar to the one given in Remark $5.3$ in \cite{BF7}).
Let us  discuss for example the case $h=2$ (the case $h=3$ is obtained by Hodge duality).

\medskip

Let $\alpha\in \mathcal{D}(\G,E_0^h)$. Arguing as in Remark $5.3$ in \cite{BF7} (formula $(48)$ therein is true in any Carnot group),  we have
\begin{equation}\label{eq.stimanormaWmpp}
\|\alpha\|_{W^{s+6,p}(\G,E_0^h)}\leq C\left(\sum_{d(I)=6}\,\|X^I\alpha\|_{W^{s,p}(\G,E_0^h)}+\|\alpha\|_{L^p(\G,E_0^h)}\right).
\end{equation}
Since $2<d(I)=6<12$, by Theorem \ref{teo.kernelCartan23}.2, there exists $\widetilde{K_I}$ a matrix valued kernel of type $6$ such that
$$X^I\alpha=\Delta_{R,h}\alpha\ast\widetilde{K_I}.$$
Thus, 
$$\|X^I\alpha\|_{W^{s,p}(\G,E_0^h)}\leq C\sum_{i,j}\,\|(\Delta_{R,h}\alpha)_j\ast (\widetilde{K_I})_{i j}\|_{W^{s,p}(\G)}.$$
We have
$$(\Delta_{R,h}\alpha)_j\ast (\widetilde{K_I})_{i j}=((d_c\delta_c)^2\alpha)_j\ast (\widetilde{K_I})_{i j}+((\delta_cd_c)^3\alpha)_j\ast (\widetilde{K_I})_{i j},$$
and
$$((d_c\delta_c)^2\alpha)_j\ast (\widetilde{K_I})_{i j}=\sum_k\,P_{i k} (d_c\delta_c\alpha)_k\ast (\widetilde{K_I})_{i j},$$
$$((\delta_cd_c)^3\alpha)_j\ast (\widetilde{K_I})_{i j}=\sum_k\,Q_{i k} (d_c\delta_cd_c\alpha)_k\ast (\widetilde{K_I})_{i j},$$
where $P_{i k}$ and $Q_{i k}$ are constant coefficient homogeneous operators of order $6$ with respect to horizontal derivatives. Thus, by \eqref{eq.convspostcose}, the term $((d_c\delta_c)^2\alpha)_j\ast (\widetilde{K_I})_{i j}$ can be written as a linear combination of
$$(d_c\delta_c\alpha)_k\ast \checkV{X^J} \checkV{(\widetilde{K_I})_{i j}}\quad\quad\text{with $d(J)=6$},$$
and the term $((\delta_cd_c)^3\alpha)_j\ast (\widetilde{K_I})_{i j}$ as a linear combination of 
$$(d_c\delta_cd_c\alpha)_k\ast \checkV{X^J} \checkV{(\widetilde{K_I})_{i j}}\quad\quad\text{with $d(J)=6$}.$$
By Proposition \ref{prop.kernelopcheck}, the kernels $\checkV{X^J} \checkV{(\widetilde{K_I})_{i j}}$ are kernels of type $0$. Therefore,  by Proposition \ref{folland zero},
\begin{gather}\label{eq.scrittsobgaffp}
\begin{split}
\|X^I\alpha\|_{W^{s,p}(\G,E_0^h)}&\leq C\sum_{i,j}\,\|(\Delta_{R,h}\alpha)_j\ast (\widetilde{K_I})_{i j}\|_{W^{s,p}(\G)}\\
&\leq C\left(\|d_c\delta_cd_c\alpha\|_{W^{s,p}(\G,E_0^{h+1})}+\|d_c\delta_c\alpha\|_{W^{s,p}(\G,E_0^{h})}\right).
\end{split}
\end{gather}
Putting together \eqref{eq.stimanormaWmpp} with \eqref{eq.scrittsobgaffp}, we complete the proof.
\end{proof}

\medskip

\begin{remark}
 With the same hypothesis we can also obtain the following estimate,
\begin{align}\label{eq.sobgaffeqineqteocorodeltad}
\|\alpha\|_{W^{s+M_h/2,p}(\G,E_0^h)}\leq C\Big(&\|d_c\alpha\|_{W^{s+M_h/2-s_h,p}(\G,E_0^{h+1})}\\
&+\|\delta_c\alpha\|_{W^{s+M_h/2-r_h,p}(\G,E_0^{h-1})}+\|\alpha\|_{L^p(\G,E_0^h)}\Big),\nonumber
\end{align}
where $s_h, r_h$ are as in \eqref{indici} and $M_h$ is the order of $\Delta_{R,h}$.   
\end{remark}

\medskip

\section*{Acknowledgments}

A.B. and A.R. are supported by the University of Bologna, funds for selected research topics, 
and by GNAMPA of INdAM (Istituto Nazionale di Alta Matematica ``F. Severi''), Italy.


\bibliographystyle{amsplain}

\bibliography{bibliography}

\bigskip
\tiny{
\noindent
Annalisa Baldi and Alessandro Rosa,
\par\noindent
Universit\`a di Bologna, Dipartimento
di Matematica\par\noindent Piazza di
Porta S.~Donato 5, 40126 Bologna, Italy.
\par\noindent
e-mail:
annalisa.baldi2@unibo.it $,\,\;$ alessandro.rosa15@unibo.it 
}

\medskip

%


\end{document}